\documentclass{amsart}
\usepackage{bbm}
\usepackage{graphicx}
\usepackage{verbatim}
\usepackage{color}

\usepackage[latin1]{inputenc}

\newtheorem{theorem}{Theorem}

\theoremstyle{definition}
\newtheorem{prop}{Proposition}
\newtheorem{lemma}{Lemma}

\newtheorem{corollary}{Corollary}

\newtheorem{assumption}{Assumption}[part]

\def\N{{\mathbb N}}
\def\R{{\mathbb R}}

\def\P{{\mathbb P}}
\def\E{{\mathbb E}}

\newcommand{\diff}{\mathop{}\mathopen{}\mathrm{d}}
\newcommand\croc[1]{\left\langle #1\right\rangle}
\newcommand\ind[1]{\mathbbm{1}_{\left\{#1\right\}}}
\newcommand{\Ept}[1]{\mathbb{E}\left(#1\right)}
\newcommand{\steq}[1]{\stackrel{\text{\rm #1.}}{=}}

\def\cal{\mathcal}
\def\eps{\varepsilon}

\def\var{\mathrm{var}}
\def\card{\mathrm{card}}
\def\Nun{{\cal N}^{U,N}}

\title{A Large Scale Analysis of Unreliable Stochastic Networks}
\author{Reza Aghajani}
\address[R. Aghajani]{Department of Mathematics, University of California San Diego, 9500 Gilman Drive, La Jolla, CA  92093, USA.}
\email{maghajani@ucsd.edu}

\address[Ph. Robert, W. Sun]{INRIA Paris, 2 rue Simone Iff, F-75012 Paris, France}
\author{Philippe  Robert}
\email{Philippe.Robert@inria.fr}
\urladdr{http://team.inria.fr/rap/robert}
\author{Wen Sun}\thanks{The author's work is supported by a public grant overseen by the French National Research Agency (ANR) as part of the ``Investissements d'Avenir'' program (reference: ANR-10-LABX-0098).}
\email{Wen.Sun@inria.fr}

\date{\today}
\keywords{Stochastic Networks with Failures; Mean-Field Models; Nonlinear Markov processes. Reliability}
\subjclass[2010]{Primary: 60J27,60K25; Secondary: 68M15}

\begin{document}

\begin{abstract}
The problem of reliability of a large distributed system is analyzed  via a new mathematical model. A typical framework is a system where a set of files are duplicated on several data servers.  When one of these servers breaks down, all copies of files stored on it are lost. In this way, repeated failures may lead to losses of files.  The efficiency of such a network is directly related to the performances of the mechanism used to duplicate files on servers. In this paper we study the evolution of the network using a natural duplication policy giving priority to the files with the least number of copies. 

We investigate the asymptotic behavior of the network when the number $N$ of servers is large.  The analysis is complicated by the  large dimension of the state space of the empirical distribution of the state of the network.  A stochastic model of the evolution of the network which has values in state space whose dimension does not depend on $N$ is introduced. Despite this description does not have the Markov property,  it turns out that it is converging in distribution, when the number of nodes goes to infinity, to a nonlinear Markov process.  The rate of decay of the network,  which is the key characteristic of interest  of these systems, can be expressed in terms of this asymptotic process.  The corresponding mean-field convergence results are established.   A lower bound on the exponential decay, with respect to time, of the fraction of the number of initial files with at least one copy is obtained.
\end{abstract}

\maketitle

\vspace{-5mm}

\bigskip

\hrule

\vspace{-3mm}

\tableofcontents

\vspace{-1cm}

\hrule

\bigskip

\section{Introduction}
The problem of reliability of a large distributed system is analyzed in the present paper via a new mathematical model. A typical framework is a system where files are duplicated on several data servers.  When a server breaks down, all copies of files stored on this server are lost but they can be retrieved if copies of the same files are stored on other servers. In the case when no other copy of a given file is present in the network, it is definitively lost. Failures of disks occur naturally in this context, these events are quite rare but, given the large number of nodes of these large systems, this is not a negligible phenomenon at all at network scale. See the measurements at Google in Pinheiro et al.~\cite{Goog}  for example.

In order to maintain copies on distant servers, a fraction of the bandwidth of each server has to be devoted to the duplication mechanism of its files to other servers. If, for a short period of time, several of the servers break down, it may happen that files will be lost for good just because all available copies were on these servers and that the recovery procedure was not completed before the last copy disappeared.  The natural critical parameters of such a distributed system with $N$ servers are the failure rate $\mu$ of servers, the bandwidth $\lambda$ allocated to duplication of a given server, and the total number of initial files $F_N$. The quantity $\lambda$ represents the amount of capacity that a server allocates to make duplication to enhance the durability of the network. If there are initially too many files in the system, the duplication capacity at each node may not be able to cope with the losses due to successive failures of servers and, therefore, a significant fraction of files will be lost very quickly. An efficient storage system should be able to maximize both the average number of files $\beta{=}F_N/N$  per server and  the durability, i.e.\   the first instant $T^N(\delta)$ when a fraction $\delta\in(0,1)$ of files which are definitely lost.

\subsection{Models with Independent  Losses of Copies and Global Duplication Capacity}
A large body of  the work in computer science in this domain has been devoted to the design and the implementation of duplication algorithms. These systems are known as distributed hash tables (DHT) which refer to the data structures used to manage these systems. They play an important role in the development of some large scale distributed systems  right now, in cloud computing for example, see Rhea et al.~\cite{rhea-05} and Rowstron and Druschel~\cite{rowstron-01} for example. Except extensive simulations, little has been done to evaluate the performances of these algorithms concerning the durability of the system. 

Several approaches have been used to investigate the corresponding mathematical models.  Simplified models using finite  birth and death processes have been often used,  see Chun et al.~\cite{chun-06}, Picconi et al.~\cite{picconi-07} and Ramabhadran and Pasquale~\cite{Rama}.  In Feuillet and Robert~\cite{Feuillet-Robert} and Sun et al.~\cite{SFR}, the authors studied how the durability $T(\delta)$ scales with the number of servers $N$ and the maximum number of copies $d$ of each file,  under simplifying assumptions on file losses and the duplication mechanism.  For this later work, each  copy of a file is assumed to be lost at a certain fixed rate, independently of the other copies of files. Secondly, they assumed that the duplication capacity can be used globally. This means that,  if each of $N$ servers has an available bandwidth $\lambda$ to duplicate files, the total capacity for duplication, $\lambda N$, can be used to create a copy of any file  in the system to another server. With these assumptions, the mathematical representation of the network is significantly simplified because it is not necessary to know the locations of copies of files to derive the dynamics of the system. In particular, in ~\cite{SFR}, a Markovian model with a fixed state space of dimension $d{+}1$ has been investigated: if for $0\leq i\leq d$ and $t\geq 0$, $X_i^N(t)$ is the number of files with $i$ copies, then the vector $X^N(t)=(X_0^N(t),X_1^N(t),\ldots,X_{d}^N(t))$ is a Markov process on $\N^{d+1}$ under the hypothesis that the global capacity $\lambda N$ is devoted to a file with the least number of copies.  They have shown  that the durability $T^N(\delta)$ is of the order $N^{d-1}$ for a large $N$ under certain conditions.  Limit theorems were established for the rescaled process $(X^N(N^{d-1}t))$ using various technical estimations. In Sun et al.~\cite{SSMRS} the impact of placement policies, i.e.\ policies determining the location of the node to make a copy of a given file, is investigated.

\subsection{Stochastic Models with Local Duplication Features}
In this paper, we consider a more realistic stochastic model for these systems, dropping the two main simplifying assumptions of previous works on copy losses and duplication capacity. 
\renewcommand{\theenumi}{\roman{enumi}}
\begin{enumerate}
  \item {\em Simultaneous losses due to server failures.} Each server can fail with a constant rate $\mu$, and independently  of the other servers. When a server fails, all copies on that server are lost simultaneously, and therefore, the copy losses are not independent anymore. This dependency and bursty losses of file copies has a crucial effect on system performance. It is assumed that a new, empty, server replaces a failed server so that the total number of nodes is constant. 

  \item {\em Local duplication capacity.} The duplication capacity is assumed to be local, that is, each server has a capacity $\lambda$ to duplicate the copies  of files present on that server. In particular, this capacity cannot be used to copy files of other servers, as it is case for models with a global duplication capacity.
  \item {\em Duplication Policy: Priority to files with smallest number of copies}. The capacity of a server is allocated to duplicating one of its own files which has the smallest number of copies  alive in the network. See Section~3 of Sun at al.~\cite{SSMRS} for a quick description of how this kind of mechanism can be implemented in practice.  It is copied,  uniformly at random, onto one of the servers which does not have such a copy.  
\end{enumerate}
Without a duplication mechanism,  it is not difficult to see that the probability that a given file with $d$ initial copies has still at least one copy at time $t$ is of the order of $d\exp(-\mu t)$ for a large $t$, when $\mu$ is the failure rate of servers. If, initially, there are $\lfloor \beta N\rfloor$ files, all with $d$ copies scattered randomly in the network, the average fraction of files with at least a copy at time $t$ is thus of the order of  $\beta d\exp(-\mu t)$. The central question is how much a duplication mechanism can improve these (poor) performances.

One cannot expect, intuitively, that the average lifetime of a file will grow significantly with $N$ as in the case of a global duplication capacity, see Sun et al.~\cite{SFR}  where the decay occurs only on the ``fast'' time scale  $t\mapsto N^{d-1}t$. In contrast, as it will be seen, the decay of our system  occurs in fact on the ``normal'' time scale $t\mapsto t$. The main aim of this paper is of investigating the exponential decay rate of the fraction of the  number of files alive at time~$t$ with bounds of the form
\[
    Ce^{-\mu\kappa t}.
\]
Of course duplication is of interest only if $\kappa{<}1$ and in fact is as close to $0$ as possible. The goal of this paper is to investigate the decay of the system described above via a mean-field approach. This is the key problem of these systems in our view. 

With these assumptions, our mathematical model turns out to have stark differences compared to previous stochastic models. For a system of fixed size $N$, the exact dynamics of the system under above duplication mechanism is quite intricate, and hence, obtaining mathematical quantitative results to estimate the coefficient $\kappa$ is quite challenging. A natural approach is of studying the performance of the system when the number of servers $N$ goes to infinity. 

To illustrate in a simpler setting the difficulties of these models, we first consider the case where there are at most two copies of each file stored on the system ($d{=}2$). In this case, a Markovian representation of the state of the system can be given by
\begin{equation}\label{def_MarkovRep2}
    (X^N(t))=(X^N_{i,j}(t), i,j=0,1,...,N),
\end{equation}
where for $1{\leq} i{\not=j}{\leq} N$, $X^N_{i,j}(t)$ is the number of files which have copies on server $i$ and $j$ at time $t$, and $X^N_{i,0}(t){=}X^N_{0,i}(t)$ is the number of files having only one copy located on server $i$. The state space of the state of a given node is therefore of dimension of the order of $N$ which does not seem to be not amenable to analysis since the dimension of the basic state space is growing with $N$.

To overcome this difficulty, we introduce a reduced state representation in which each node $i$ is described by only two variables: the number of files whose unique copy is on server $i$ and the number of files with two copies and one of the copies is on $i$. The state of a node is then a two-dimensional state space. The empirical distribution associated with such a representation has values in a state space of probability distributions on $\N^2$.  This dimension reduction comes nevertheless at a price, the loss of the Markov property.  We prove that this non-Markovian description of the network is in fact converging in distribution, as $N$ goes to infinity, to {\em a nonlinear Markov process}, $(\overline{R}(t)){=}(\overline{R}_1(t),\overline{R}_2(t)){\in}\N^2$ satisfying the following Fokker-Planck Equations
\begin{multline}\label{FK}
    \frac{\diff}{\diff t} \Ept{\rule{0mm}{4mm}f(\overline{R}_1(t),\overline{R}_2(t))}=
    \lambda \Ept{\left(\rule{0mm}{4mm}f(\overline{R}(t){+}e_2{-}e_1)- f(\overline{R}(t))\right)\ind{\overline{R}_1(t)>0}}\\
    +\lambda \P(\overline{R}_1(t)>0)\Ept{\rule{0mm}{4mm}f(\overline{R}(t){+}e_2)- f(\overline{R}(t))}\\
    +\mu \Ept{\rule{0mm}{4mm}f(0,0)-f(\overline{R}(t))}
    +\mu \Ept{\left(\rule{0mm}{4mm}f(\overline{R}(t){+}e_1{-}e_2)- f(\overline{R}(t))\right)\overline{R}_2(t)},
\end{multline}
with $e_1{=}(1,0)$ and $e_2{=}(0,1)$, and $f$ is a function with finite support on $\N^2$. In this setting the asymptotic fraction of the number of files alive at time $t$ is given by $\E(\overline{R}_1(t)){+}\E(\overline{R}_2(t))/2$.

The asymptotic process $(\overline{R}(t)=(\overline{R}_1(t),\overline{R}_2(t))$ is a jump process with a type of jump, $x\mapsto x{+}e_2$ having  time-dependent and distribution-dependent rate given by $\lambda\P(\overline{R}_1(t){>}0)$ which is the nonlinear term of this evolution equation. 

\subsection{Rate of Convergence to Equilibrium}
It will be shown that the  non-time-homogeneous Markov process defined by~Relation~\eqref{FK}  is converging to the unique distribution $\pi{=}\delta_{(0,0)}$, the Dirac measure at $(0,0)$, corresponding to a system with all files lost. The decay of the network is thus closely related to the convergence rate to equilibrium of this Markov process.

As we have seen before, the problem is of finding  a constant $\kappa{>}0$ for which  the asymptotic fraction of the number of files alive at time $t$ has an exponential decay with parameter $\mu\kappa$, i.e.\ 
\begin{equation}\label{kap}
\E\left(\overline{R}_1(t)\right)+\frac{1}{2}\E\left(\overline{R}_2(t)\right)\leq \left(\overline{R}_1(0)+\frac{1}{2}\overline{R}_2(0)\right) e^{-\mu\kappa t}, \quad \forall t\geq 0.
\end{equation}
The convergence rate can be defined in terms of the {\em Wasserstein  distance} between the distribution $P_t$ of the distribution at time~$t$ and the equilibrium distribution~$\pi$, 
\[
{\cal W}_1(P_t,\pi)\steq{def} \inf\left\{\E(\diff(X,Y)): X{\steq{dist}}P_t, Y{\steq{dist}}\pi\right\},
\]
where $\diff(\cdot,\cdot)$ is some distance on the state space. One has  to find the best possible constant $\alpha$ such that the relation 
\begin{equation}\label{eqq1}
{\cal W}_1(P_t,\pi)\leq {\cal W}_1(P_0,\pi)e^{-\alpha t}
\end{equation}
holds for all $t{\geq}0$.

For {\em time-homogeneous}, i.e.\ ``standard'', Markov processes, this is already a difficult problem. For finite state spaces, tight estimates are known for some classical random walks, see~Aldous and Diaconis~\cite{Aldous} for example. When the state space is countable, results are more scarce. Lyapunov functions techniques to prove the existence of finite exponential moments of hitting times of finite sets can give a lower bound on the exponential decay $\alpha$. This is, in general, a quite rough estimate for $\alpha$, furthermore it does not give an estimate of the form~\eqref{eqq1}. See   Section~6.5 of Nummelin~\cite{Nummelin}, see also Chapter~15 of Meyn and Tweedie~\cite{Meyn}.

In the continuous case, i.e.\ with Brownian motions instead of Poisson processes, some functional  inequalities  have been successfully used to obtain Relations of the form~\eqref{eqq1}, see Markowich and Villani~\cite{Villani} and Desvillettes and Villani~\cite{Villani2} for surveys on this topic.  An extension of this approach for the case of the discrete state space turns out to be more difficult to do. Some generalizations  have been proposed by Caputo et al.~\cite{Caputo}, Joulin~\cite{Joulin} and Ollivier~\cite{Ollivier} for some jump processes. They have been used with some success, see Alfonsi et al.~\cite{Jourdain} and Thai~\cite{Thai} for example. For classical birth and death processes on $\N$,  it leads to  some quite specific (and strong)  conditions on the birth and death rates in order to have a positive exponential decay $\alpha$.  

For {\em non-linear} Markov processes, which is our case, the situation is, of course, much more complicated.  Recall that, in this context,  there may be  several invariant distributions, so that convergence to equilibrium is a more delicate notion. Note that this is not our case however.   Ideas using the functional  inequalities mentioned before have been also used but for specific stochastic models. See Carrillo et al.~\cite{Carrillo} and Cattiaux et al.~\cite{Cattiaux}  for a class of diffusions and Thai~\cite{Thai} for a class of  birth and death processes.  They do not seem of any help for the class of models we consider. To the best of our knowledge,  the results of optimality concerning the exponential decay $\alpha$ are quite scarce. Only lower  bounds are provided in general. For  the non-linear Markov processes associated to the mean-field results of this paper, our approach will mainly use some monotonicity properties  to derive lower bounds on the exponential decay. 

In a first step, the present paper develops a mathematical framework to get  a convenient  asymptotic description of our network, Equations~\eqref{FK}, and, secondly, to obtain explicit lower bounds on its exponential decay. This program is completed in the case $d{=}2$. In particular it is shown in Proposition~\ref{up-prop} of Section~\ref{Decay-Sec} that  Equation~\eqref{kap} holds some a constant $\kappa{=}\kappa_2^+$. Note that, however, as it can be expected in such a complicated context, we are not able to show that the constant $\kappa_2^+$ is optimal.  As it will be seen, the case $d{>}2$ is more awkward in terms of an asymptotic picture, but results on the exponential decay of the network can be nevertheless  obtained  by studying a non-linear Markov process dominating, in some sense, the original Markov process. 

\subsection*{Outline of the Paper}
Section~\ref{Mod-sec} introduces the main evolution equations of the state of the network. Section~\ref{McV-sec} investigates the existence and uniqueness properties of a nonlinear Markov process, the main result is Theorem~\ref{McKeanTheo}. This process turns out to be the limit of a reduced description of the network. Section \ref{MF-sec} shows the mean-field convergence of the reduced description of the network to this asymptotic process, this is Theorem \ref{MF-Theo}. Section \ref{Decay-Sec} studies the asymptotic behavior of the nonlinear Markov process. A lower bound on the exponential decay, with respect to time, of the fraction of the number of initial files still alive at time $t$ is derived.  These results are obtained in the case when the maximal number of copies for a given file is $2$.  Section~\ref{Ext-Sec}  investigates the case of a general $d$. By using a simpler stochastic model, for which a mean-field limit result also holds, a multi-dimensional analogue of the set of equations~\eqref{FK} is introduced and analyzed.   It  gives a lower bound on the exponential decay of the number of files present in the network.  It is expressed as the maximal eigenvalue of a $d{\times}d$ matrix. The proofs of the main results rely essentially on several ingredients: careful stochastic calculus with marked Poisson processes, several technical estimates, Lemmas~\ref{Tech_set} and~\ref{lem_bound}, and mean-field techniques.

\section{The Stochastic Model}\label{Mod-sec}
In this section, we describe the dynamics of our system. Recall that the system has $N$ servers, and until Section~\ref{Ext-Sec}, it is assumed that each file has at most two copies in the system. Recall the Markovian representation $(X^N(t))$ defined in \eqref{def_MarkovRep2}, that is,
\[
    X^N(t)=(X_{i,j}^N(t), 1\leq i, j\leq N),
\]
where, for $1{\leq} i{\leq} N$, $X_{i,0}^N(t)$ is the number of files whose unique copy is located in server~$i$ at time $t$,  and $X_{i,j}^N(t)$ is the number of files with a copy on  server $i$ and on server $j$, $1{\leq} j{\leq} N$, $j{\not=}i$. Note the symmetry $X_{i,j}^N(t){=}X_{j,i}^N(t)$, and by convention, $X_{i,i}^N(\cdot){\equiv}0$. We assume that all files have initially two copies and are randomly scattered on the network, as described below.\\

\begin{assumption}\label{asm_initial}
(Initial State) For  $1{\leq}i{\leq}N$, there are $A_i$ files on server $i$ and each file $\ell{=}1$,\ldots ,$A_i$ has another copy on server $V_{i,\ell}$, where
\begin{itemize}
    \item $A_i$,$i{=}1$,\ldots,$N$, are i.i.d.\ square  integrable random variables on $\N$,
    \item For each $i$, $V_{i,\ell}^N,\ell{\geq}1$, are i.i.d.\ random variables with uniform distribution over  $\{1,\ldots,N\}{\setminus}\{i\}$.
\end{itemize}
Hence,  $X_{i,0}^N(0){=}0$ and
\[
    X_{i,j}^N(0)= \sum_{\ell=1}^{A_i} \ind{V_{i,\ell}^N=j}+\sum_{\ell=1}^{A_j} \ind{V_{j,\ell}^N=i}.
\]
\end{assumption}
The total number of initial files is therefore $ F_N{\steq{def}}  A_1{+}A_2{+}\cdots{+}A_{N}$, and the initial average load of the system is
\[\beta\steq{def}\lim_{N\to+\infty} \frac{F_N}{N}=\Ept{A_1}.\]
The initial mean number of copies of files per server is therefore $2\beta$.

The initial state described in Assumption~\ref{asm_initial} have two main properties. First, it is exchangeable, in the sense that the distribution of $X^N(0)$ is invariant under permutations of server indices, and second, the two copies of each file are uniformly distributed over all servers. Alternatively, one can also assume that the total number of files $F_N$ is a fixed number, without changing the results of the paper.

\subsection*{Transitions of the State Representation}

The transitions of the Markov process $(X^N(t))$ is governed by server failures and file duplications, as described below. Throughout this paper, $f(t-)$ denotes the left limit of a function $f$ at time $t$.

\begin{itemize}
    \item \textit{Server failure. } Each server $i$ breaks down after an exponential time with parameter $\mu$. At each breakdown, all  copies on server $i$ are lost, and the server restarts immediately but empty. It is in fact replaced by a new one.  If a breakdown happens at time $t$,
        \[
            \begin{cases}
                X_{i,j}^N(t)=0,\quad \text{for all } j=1,..,N, \\[2mm]
                X_{i,0}^N(t)=0,\\[2mm]
                X_{j,0}^N(t)=X_{j,0}^N(t{-}){+}X_{i,j}^N(t{-}),\quad \text{for all } j\neq i.
            \end{cases}
        \]

      \item \textit{Duplication.}   If there are files on server $i$ with only one copy (i.e.\ $X_{i,0}^N(t){>}0$), one of such files is copied at a rate $\lambda$ onto another server $j\in\{1,\ldots,N\}{\setminus}\{i\}$ chosen uniformly at random.  If the duplication is completed at time $t$,
        \[
            \begin{cases}
                X_{i,0}^N(t)=X_{i,0}^N(t{-}){-}1,\\[2mm]
                X_{i,j}^N(t)=X_{i,j}^N(t{-}){+}1.
            \end{cases}
        \]
\end{itemize}
Note that $(X^N(t))$ is a transient Markov process converging to the state with all coordinates being $0$ (all copies are lost).

\subsection*{Stochastic Evolution Equations}
We can describe the dynamics of $(X^N(t))$ using stochastic integrals with respect to Poisson processes. Throughout the paper, we use the following notations for Poisson processes.

\begin{itemize}
  \item ${\cal N}_{\xi}$ is a Poisson point process on $\R_+$ with parameter $\xi>0$ and $({\cal N}_{\xi,i})$ is an i.i.d.\ sequence of such processes.
  \item $\overline{{\cal N}}_{\xi}$ is a Poisson point process on $\R_+^2$ with intensity $\xi\diff t\diff h$ and $(\overline{{\cal N}}_{\xi,i})$ is an i.i.d.\ sequence of such processes.
  \item For $1{\leq} j {\leq} N$, the random variable ${\cal N}^{U,N}_{\lambda,j}{=}(t_n^j,U_n^j)$ is a marked Poisson process,  $(t_n^j)$ is a Poisson process on $\R_+$  with rate $\lambda$ and $(U_n^j)$ is an i.i.d.\ random variables with uniform distribution over $\{1,\ldots,N\}{\setminus}\{j\}$. In particular, for $1{\leq}i{\leq}N$, $(\Nun_{\lambda,j}(\cdot,\{i\}))$ is an i.i.d.\  sequence of Poisson processes with rate $\lambda/(N{-}1)$.
\end{itemize}
With a slight abuse of notation, we denote $\mathcal{N}_{\lambda,j}{\steq{def}} \Nun_{\lambda,j}(\cdot,\mathbb{N})$, which is a Poisson process with rate $\lambda.$ See Kingman~\cite{Kingman} and~\cite{Bremaud} for an introduction on ordinary and marked Poisson processes. All Poisson processes used are assumed to be independent.

For every $j{=}1,...,N$, failure times of server $j$ are given by the epoch times of a Poisson process $\mathcal{N}_{\mu,j}$. A  marked Poisson process $\Nun_{\lambda,j}$ captures duplications of files from server $j$ as follows: for $n{\geq}1$, at the $n$th event time $t_n^j$ of $\Nun_{\lambda,j}$, if $X_{j,0}^N(t_n^j{-}){>}0$, a file on server $j$ is copied onto the  server whose index is given by the mark $U_n^j$.

The process $(X^N(t))$ can then be characterized as the solution of the following system of  stochastic differential equations (SDEs): for $1{\leq} i, j{\leq} N$, $j{\not=}i$ and $t{\geq} 0$,
\begin{multline}\label{Xij_diff}
    \diff X_{i,j}^N(t)= -X_{i,j}^N(t{-})\left[\rule{0mm}{4mm}{\cal N}_{\mu,i}(\diff t)+ {\cal N}_{\mu,j}(\diff t)\right]\\+ \ind{X_{j,0}^N(t{-})>0}\Nun_{\lambda,j}(\diff t,\{i\})  +\ind{X_{i,0}^N(t{-})>0}\Nun_{\lambda,i}(\diff t,\{j\}),
\end{multline}
and
\begin{multline}\label{Xi0_diff}
    \diff X_{i,0}^N(t)= -X_{i,0}^N(t{-}){\cal N}_{\mu,i}(\diff t)- \ind{X_{i,0}^N(t{-})>0}{\cal N}_{\lambda,i}(\diff t)
    \\+\sum_{j=1}^N X_{i,j}^N(t-){\cal N}_{\mu,j}(\diff t).
\end{multline}
Classical results on Poisson processes show that the process
\begin{equation*}
\left( \int_0^t \ind{X_{i,0}^N(s{-})>0}\Nun_{\lambda,i}(\diff s, \{j\})-\frac{\lambda}{N-1} \int_0^t\ind{X_{i,0}^N(s)>0}\,\diff s\right)
\end{equation*}
is a martingale whose previsible increasing process is
\[
\left(\frac{\lambda}{N-1} \int_0^t\ind{X_{i,0}^N(s)>0}\,\diff s\right).
\]
See, for example, Section~4 of Chapter~IV of Rogers and Williams~\cite{Rogers}. Therefore, for $1{\leq} i{\not=}j{\leq} N$,
\begin{multline}\label{Xij}
    X_{i,j}^N(t)=X_{i,j}^N(0) -2\mu\int_0^t X_{i,j}^N(s)\,\diff s\\+ \frac{\lambda}{N{-}1}\int_0^t \left(\ind{X_{i,0}^N(s)>0}+\ind{X_{j,0}^N(s)>0}\right)\,\diff s+{\cal M}_{i,j}^N(t),
\end{multline}
and
\begin{multline}\label{Xi0}
    X_{i,0}^N(t)=X_{i,0}^N(0) -\mu\int_0^t X_{i,0}^N(s)\,\diff s-\lambda\int_0^t \ind{X_{i,0}^N(s)>0}\,\diff s
    \\+\mu\sum_{j=1}^N \int_0^t X_{i,j}^N(s)\,\diff s+{\cal M}_{i,0}^N(t),
\end{multline}
where $({\cal M}_{i,0}^N(t))$, $1{\leq} i{\leq} N$, and $({\cal M}_{i,j}^N(t))$, $1{\leq} i{<}j{\leq} N$, are local martingales with the respective  previsible increasing processes:
\begin{multline*}
 \croc{{\cal M}^N_{i,j}}(t)= 2\mu\int_0^t \left(X_{i,j}^N(s)\right)^2\,\diff s\\+\frac{\lambda}{N{-}1}\int_0^t \ind{X_{i,0}^N(s)>0}\,\diff s +\frac{\lambda}{N{-}1}\int_0^t \ind{X_{j,0}^N(s)>0}\,\diff s,
\end{multline*}
and
\begin{multline*}
    \croc{{\cal M}^N_{i,0}}(t)= \mu\int_0^t \left(X_{i,0}^N(s)\right)^2\,\diff s\\+\frac{\lambda}{N{-}1}\int_0^t \ind{X_{i,0}^N(s)>0}\,\diff s
    +\mu \int_0^t\sum_{j=1}^N \left(X_{i,j}^N(s)\right)^2\, \diff s.
\end{multline*}

\section{An Asymptotic Process} \label{McV-sec}
As mentioned in the introduction, for $1{\leq}i{\leq}N$, the state of each server $i$ at time~$t$ can alternatively be described by the pair 
\begin{equation}\label{defR}
R_i^N(t)=(R_{i,1}^{N}(t),R_{i,2}^{N}(t)),
\end{equation}
where $R_{i,1}^N(t)$ (resp.\  $R_{i,2}^N(t)$) is the number of files with one copy  (resp.\ two copies) at node $i$.  This reduced representation can be obtained from the full Markovian representation $(X^N(t))$ via
\[
    R_{i,1}^{N}(t)=X_{i,0}^N(t) \quad \text{ and }\quad R_{i,2}^N(t)=\sum_{j=1}^N X_{i,j}^N(t).
\]
Therefore, the evolution equations of $(R^N_i(t))$ can be deduced from the SDEs \eqref{Xij_diff} and~\eqref{Xi0_diff}:
\begin{multline}\label{eqR1}
  \diff R_{i,1}^N(t)  = - R_{i,1}^N(t-){\cal N}_{\mu,i}(\diff t) \\- \ind{R_{i,1}^N(t-)>0}{\cal N}_{\lambda,i}(\diff t) +\sum_{j\not=i}X_{i,j}^N(t-){\cal N}_{\mu,j}(\diff t)
\end{multline}
\begin{multline}\label{eqR2}
    \diff R_{i,2}^N(t)  = - R_{i,2}^N(t-){\cal N}_{\mu,i}(\diff t)+\ind{R_{i,1}^N(t-)>0}{\cal N}_{\lambda,i}(\diff t)
   \\ -\sum_{j\not=i}X_{i,j}^N(t-){\cal N}_{\mu,j}(\diff t)+\sum_{j\not=i}\ind{R^N_{j,1}(t-)>0}\Nun_{\lambda,j}(\diff t,\{i\}).
\end{multline}
The process $(R^N(t)){=}(R_{i}^N(t), 1{\leq} i{\leq} N)$ lives on a state space of dimension $2N$ instead of $N^2$. The process $(R^N(t))$ still captures the information on the decay of the system since, for example, the total  number of files which are still available in the network  at time $t$ can be expressed as
\[
\sum_{i=1}^NR_{i,1}^N(t) + \frac{1}{2} R_{i,2}^N(t). 
\]
 This dimension reduction comes at the price of the loss of the Markov property. The evolution equations of $(R^N_i(t))$ are not autonomous, they depend on the process $(X^N(t))$, and, consequently, the process $(R^N(t))$  does not have the Markov property. However, as it will be seen, the  limit in distribution of  $(R_{i,1}^N(t),R_{i,2}^N(t))$ turns out to be a nonlinear Markov process, or a so-called McKean-Vlasov process. See e.g.\ Sznitman~\cite{Sznitman}. In this section we characterize this limiting process, while the proof of convergence as $N$ goes to infinity is given in the next section.

\subsection*{An Intuitive Introduction of the Asymptotic Process}

The purpose of this section is only of motivating the asymptotic process;  rigorous arguments to establish the convergence results are given later. Fix some  $1{\leq}i{\leq}N$ and assume for the moment that $(R_{i,1}^N(t), R_{i,2}^N(t))$ is converging in distribution to a process $(\overline{R}_{1}(t),\overline{R}_{2}(t))$. Define the positive random measure
\[
    {\cal P}_i^N([0,t])\stackrel{\text{def.}}{=}\int_0^t \sum_{j\not=i}X_{i,j}^N(s-){\cal N}_{\mu,j}(\diff s).
\]
It will be shown later in Lemma~\ref{Tech_set} that for a fixed $1{\leq}i{\leq}N$, with high probability when $N$ is large, all the variables $(X_{i,j}^N(t),1{\leq}j{\leq}N$) are  either $0$ or $1$  on a fixed time interval. In particular,  ${\cal P}_i^N$ is asymptotically a counting process, i.e.\ an increasing process with jumps of size $1$, with compensator given by 
\[
    \mu \int_0^t \sum_{j\not=i}X_{i,j}^N(s)\,\diff s=\mu\int_0^t R_{i,2}^N(s)\,\diff s.
    \]
See Jacod~\cite{Jacod} or Kasahara and Watanabe~\cite{Kasahara} for example. 
The convergence in distribution of the process $(R_{i,2}^N(s))$ to  $(\overline{R}_{2}(s))$ and standard results on convergence of point processes give  that ${\cal P}_i^N$ converges to ${\cal P}^\infty$,  an inhomogeneous Poisson process with intensity $(\overline{R}_{2}(t))$ which can be represented as
\[
    {\cal P}^\infty( \diff  t) = \int_{\R_+}\ind{0\leq h\leq \overline{R}_{2}(t-)}\overline{{\cal N}}_{\mu}(\diff t, \diff h).
    \]
See e.g.\ Kasahara and Watanabe~\cite{Kasahara} and Brown~\cite{Brown}.  Recall  that $\overline{{\cal N}}_{\mu}$ is a Poisson process on $\R_+^2$ with intensity $\mu\diff t\diff h$ (see Section~\ref{Mod-sec}). By formally taking the limit on both sides of Equation~\eqref{eqR1} as $N$ gets large.  This yields that the process $(\overline{R}_{1}(t),\overline{R}_{2}(t))$ satisfies the relation
\begin{multline}\label{eqRb1}
    \diff \overline{R}_{1}(t)= - \overline{R}_{1}(t-){\cal N}_{\mu}(\diff t) - \ind{\overline{R}_{1}(t-)>0}{\cal N}_{\lambda}(\diff t)
\\    +\int_{\R_+} \ind{0\leq h\le\overline{R}_{2}(t-)}\overline{\cal N}_{\mu}(\diff t, \diff h).
\end{multline}
A similar work can be done with Equation~\eqref{eqR2}. Consider the counting measure
\[
    {\cal Q}_i^N([0,t])\stackrel{\text{def.}}{=}\int_0^t \sum_{j\not=i}\ind{R_{j,1}^N(s-)>0}{\cal N}_{\lambda,j}^{U,N}(\diff s, \{i\}),
\]
which has the compensator, see Jacod~\cite{Jacod},
\[
    \lambda \int_0^t \frac{1}{N-1} \sum_{j\not=i}\ind{R_{j,1}^N(s)>0} \diff s.
\]
Again formally, it follows from the asymptotic independence of different servers and the law of large numbers limit for the processes $(R_{j,1}^N(t))$  that
\[
    \lim_{N\to+\infty}\left(\frac{1}{N-1} \sum_{j\not=i}\ind{R_{j,1}^N(t)>0}\right)= \left(\P\left(\overline{R}_1(t)>0\right)\right),
\]
and therefore, ${\cal Q}_i^N$ converges in distribution to an inhomogeneous Poisson process with intensity $(\P(\overline{R}_1(s){>}0))$. Therefore, taking limit from both sides of Equation~\eqref{eqR2} as $N$ gets large. One obtains that the process $(\overline{R}_{2}(t))$ satisfies
\begin{multline}\label{eqRb2}
    \diff \overline{R}_{2}(t)= - \overline{R}_{2}(t-){\cal N}_{\mu}(\diff t) + \ind{\overline{R}_{1}(t-)>0}{\cal N}_{\lambda}(\diff t)  \\   - \int_{\R_+} \ind{0\le h\le \overline{R}_{2}(t-)}\overline{{\cal N}}_{\mu}(\diff t,\diff h) +\int_{\R_+} \ind{0\leq h\le\P(\overline{R}_{1}(t)>0)}\overline{{\cal N}}_{\lambda}(\diff t, \diff h).
\end{multline}
The first result establishes the existence and uniqueness of a stochastic process satisfying the SDEs~\eqref{eqRb1} and~\eqref{eqRb2}. For $T>0$, let ${\cal D}_T\stackrel{\text{def.}}{=}{\cal D}([0,T],\N^2)$ be the set of c\`{a}dl\`{a}g functions from $[0,T]$ to $\mathbb{N}^2$ and $d_T(\cdot,\cdot)$ denotes a distance associated with the Skorohod topology on ${\cal D}_T$; see Chapter~3 of Billingsley~\cite{Billingsley}.

\begin{theorem}(McKean-Vlasov Process)\label{McKeanTheo}
For every $(x,y)\in\N^2$, the equations
\begin{equation}\label{McKean}
\begin{cases}
   & \overline{R}_{1}(t) {=}\displaystyle x{-}\int_0^t \overline{R}_{1}(s{-}){\cal N}_{\mu}(\diff s){-}\int_0^t \ind{\overline{R}_{1}(s-)>0}{\cal N}_{\lambda}(\diff s)\\ & \hspace{40mm}  \displaystyle  {+} \iint_{[0,t]\times\R_+}\ind{0\le h\le\overline{R}_{2}(s-)}\overline{{\cal N}}_{\mu}(\diff s,\diff h ),\\[5mm]
  &  \overline{R}_{2}(t)\displaystyle{=}y{-}\int_0^t\overline{R}_{2}(s{-}){\cal N}_{\mu}(\diff s)- \iint_{[0,t]\times\R_+}\ind{0\le h\le \overline{R}_{2}(s-)}\overline{{\cal N}}_{\mu}(\diff s,\diff h)
    \\&\displaystyle {+}\int_0^t\ind{\overline{R}_{1}(s-)>0}{\cal N}_{\lambda}(\diff s){+}\iint_{[0,t]\times\R_+}\ind{0\le h\le \P(\overline{R}_{1}(s)>0)}\overline{{\cal N}}_{\lambda}(\diff s,\diff h).
\end{cases}
\end{equation}
have a unique solution $(\overline R_1(t),\overline R_2(t))$ in ${\cal D}_T$.
\end{theorem}

The set of probability distributions on ${\cal D}_T$ is denoted as ${\cal P}({\cal D}_T)$. Theorem \ref{McKeanTheo} states that there exists a unique $\pi\steq{dist}(\overline R_1(t),\overline R_2(t))$ in ${\cal P}({\cal D}_T)$ which satisfies Equation~\eqref{McKean}. See Rogers and Williams~\cite{Rogers} for definitions of existence and uniqueness of  a solution. Note that the solution to Equation~\eqref{McKean} solves the Fokker-Planck Equation~\eqref{FK} of the introduction.

\begin{proof}
Define the uniform norm $\|\cdot\|_{\infty,T}$ on ${\cal D}_T$, if $f=(f_1,f_2)\in{\cal D}_T$,
\[
    \|f\|_{\infty,T}=\sup\{\|f(t)\|:0\leq t\leq T\}=\sup\{|f_1(t)|+|f_2(t)|:0\leq t\leq T\}.
\]
One can introduce the  {\em Wasserstein metrics} on ${\cal P}({\cal D}_T)$ as follows, for $\pi_1,\pi_2\in {\cal P}({\cal D}_T)$
\begin{align}
    W_T(\pi_1,\pi_2)&=\inf_{\pi \in {\cal C}_T(\pi_1,\pi_2)}\int_{\omega=(\omega_1,\omega_2)\in{\cal D}_T^2} \left[d_T(\omega_1,\omega_2)\wedge 1\right]\, \diff\pi(\omega),\\
    \rho_T(\pi_1,\pi_2)&=\inf_{\pi \in {\cal C}_T(\pi_1,\pi_2)}\int_{\omega=(\omega_1,\omega_2)\in{\cal D}_T^2} \left[\|\omega_1-\omega_2\|_{\infty,T}\wedge 1\right] \,\diff\pi(\omega),
\end{align}
where $a\wedge b=\min\{a,b\}$ for $a$, $b\in\R$ and  $ {\cal C}_T(\pi_1,\pi_2)$ is the subset of couplings of $\pi_1$ and $\pi_2$, i.e.\  the subset of ${\cal P}({\cal D}_T{\times}{\cal D}_T)$ whose first (resp. second) marginal is $\pi_1$ (resp. $\pi_2$). Since $({\cal D}_T,d_T)$ is separable and complete, the space $({\cal P}({\cal D}_T),W_T)$ is complete, which gives the topology of  convergence in distribution on ${\cal P}({\cal D}_T)$. Clearly, for any $\pi_1,\pi_2\in {\cal P}({\cal D}_T)$, one has the relation  $W_T(\pi_1,\pi_2)\le\rho_T(\pi_1,\pi_2)$.

Let $\Psi: ({\cal P}({\cal D}_T),W_T){\to} ({\cal P}({\cal D}_T),W_T)$ be the mapping that takes $\pi$ to the distribution $\Psi(\pi)$ of $R_\pi$, where $(R_\pi(t)){=}({R}_{\pi,1}(t),{R}_{\pi,2}(t))$ is the unique solution to the SDEs
\begin{multline*}
    {R}_{\pi,1}(t) {=}x\displaystyle -\int_0^t {R}_{\pi,1}(s-){\cal N}_{\mu}(\diff s){-} \int_0^t\ind{{R}_{\pi,1}(s-)>0}{\cal N}_{\lambda}(\diff s)
\\    \displaystyle   + \iint_{[0,t]\times\R_+}\ind{0\le h\le{R}_{\pi,2}(s-)}\overline{{\cal N}}_{\mu}(\diff s,\diff h ),
\end{multline*}
\begin{multline*}
  {R}_{\pi,2}(t)  \displaystyle{=}y-\int_0^t {R}_{\pi,2}(s{-}){\cal N}_{\mu}(\diff s) - \iint_{[0,t]\times \R_+}\ind{0\le h\le {R}_{\pi,2}(s-)}\overline{{\cal N}}_{\mu}(\diff s,\diff h) 
\\ {+} \int_0^t\ind{{R}_{\pi,1}(s-)>0}{\cal N}_{\lambda}(\diff s)     \displaystyle{+}\iint_{[0,t]\times\R_+}\ind{0\le h\le \pi(r_1(s)>0)}\overline{{\cal N}}_{\lambda}(\diff s,\diff h),
\end{multline*}
with initial condition $({R}_{\pi,1}(0),{R}_{\pi,2}(0)){=}(x,y)$. Note that
\[
    \pi(r_1(t)>0)=\int_{\omega=(r_1,r_2)\in{\cal D}_T} \ind{r_1(t)>0}\diff \pi(\omega).
\]
The existence and uniqueness of a solution to Equations~\eqref{McKean} is equivalent to the existence and uniqueness of a fixed point $\pi{=}\Psi(\pi)$.

For any $\pi_a,\pi_b\in {\cal P}({\cal D}_T)$ then, let $R_{\pi_a}$ and $R_{\pi_b}$ both be solutions to the equations of the display above driven by same Poisson processes. Therefore, the distribution of the pair $(R_{\pi_a}(t),R_{\pi_b}(t))$ is a coupling of $\Psi(\pi_a)$ and $\Psi(\pi_b)$, and hence,
\begin{equation}\label{rho}
    \rho_t(\Psi(\pi_a),\Psi(\pi_b))\leq \Ept{\|R_{\pi_a}-R_{\pi_b}\|_{\infty,t}}.
\end{equation}
For $t{\le} T$, using the definition of $R_{\pi_a}$ and $R_{\pi_b}$,
\begin{multline}\label{dif}
    \|R_{\pi_a}- R_{\pi_b}\|_{\infty,t}{=}\sup_{s\leq t}\left(\left|R_{\pi_a,1}(s)-R_{\pi_b,1}(s)\right|{+}\left|R_{\pi_a,2}(s)-R_{\pi_b,2}(s)\right|\right)\\
    {\le}  \int_0^t\left(\left|R_{\pi_a,1}(s{-}){-}R_{\pi_b,1}(s{-})\right|{+}\left|R_{\pi_a,2}(s{-}){-}R_{\pi_b,2}(s{-})\right|\right)\cal{N}_\mu(\diff s)\\
    {+} 2\int_0^t\left|\ind{R_{\pi_a,1}(s{-}){>}0}-\ind{R_{\pi_b,1}(s{-}){>}0}\right|\cal{N}_{\lambda}(\diff s)\\
    {+} 2\int_0^t\int_0^\infty\ind{ R_{\pi_a,2}(s{-})\wedge R_{\pi_b,2}(s{-}) \le h\le R_{\pi_a,2}(s{-})\vee R_{\pi_b,2}(s{-})}
    \overline{\cal{N}}_{\mu}(\diff s,\diff h)\\
    {+}\int_0^t\int_0^\infty\ind{ \pi_{a}(r_1(s){>}0)\wedge \pi_{b}(r_1(s){>}0)\le h\le \pi_{a}(r_1(s){>}0)\vee\pi_{b}(r_1(s){>}0)}
    \overline{\cal{N}}_{\lambda}(\diff s,\diff h).
\end{multline}
We bound the expected value of each of the terms of the right-hand side above. First, for $\ell{=}1$, $2$,
\begin{align}
    \Ept{\int_0^t\left|R_{\pi_a,\ell}(s{-}){-}R_{\pi_b,\ell}(s{-})\right|\cal{N}_\mu(\diff s)}& =
    \mu\Ept{\int_0^t\left|R_{\pi_a,\ell}(s){-}R_{\pi_b,\ell}(s)\right|\,\diff s} \tag{a} \\
    &\leq \mu\int_0^t\Ept{ \left\|R_{\pi_a}-R_{\pi_b}\right\|_{\infty,s}}\,\diff s. \notag 
\end{align}
For the second term on the right-hand side of \eqref{dif}, since $R_{\pi_a,1}(s)$ and $R_{\pi_b,1}(s)$ are integer valued,
\[
    \left|\ind{R_{\pi_a,1}(s){>}0}-\ind{R_{\pi_b,1}(s){>}0}\right|\leq |R_{\pi_a,1}(s)-R_{\pi_b,1}(s)|,
\]
and hence, using (a), we have the bound
\begin{multline}
  \Ept{\int_0^t\left|\ind{R_{\pi_a,1}(s{-}){>}0}-\ind{R_{\pi_b,1}(s{-}){>}0}\right|\cal{N}_{\lambda}(\diff s)} \\\leq \mu\int_0^t\Ept{\left\|R_{\pi_a}-R_{\pi_b}\right\|_{\infty,s}}\,\diff s.\tag{b}
\end{multline}
Similarly, for the third term on the right-hand side of \eqref{dif}, we have
\begin{multline}\tag{c}
    \Ept{\int_0^t\int_0^\infty\ind{ R_{\pi_a,2}(s{-})\wedge R_{\pi_b,2}(s{-}) \le h\le R_{\pi_a,2}(s{-})\vee R_{\pi_b,2}(s{-})}
    \overline{\cal{N}}_{\mu}(\diff s,\diff h)}\\
 =\mu \int_0^t\Ept{|R_{\pi_a,2}(s){-} R_{\pi_b,2}(s)|}\,\diff s 
\leq \mu\int_0^t\Ept{\left\|R_{\pi_a}{-} R_{\pi_b}\right\|_{\infty,s}}\,\diff s.
\end{multline}
Finally, for the last term on the right-hand side of \eqref{dif},
\begin{multline}
    \Ept{\int_0^t\int_0^\infty\ind{ \pi_{a}(r_1(s){>}0)\wedge \pi_{b}(r_1(s){>}0)\le h\le \pi_{a}(r_1(s){>}0)\vee\pi_{b}(r_1(s){>}0)}
    \overline{\cal{N}}_{\lambda}(\diff s,\diff h)}\\=
    \lambda \int_0^t\left|\pi_{a}(r_1(s){>}0){-} \pi_{b}(r_1(s){>}0)\right|\,\diff s.
    \tag{d}
\end{multline}
Note that for every coupling $\pi{\in}{\cal C}_T(\pi_a,\pi_b)$ of $\pi_a$ and $\pi_b$,
\begin{align*}
\int_0^t|\pi_{a}(r_1(s){>}0)&- \pi_{b}(r_1(s){>}0)\,\diff s \\& =
    \int_0^t\left|\pi\left((r^a,r^b):r_{1}^a(s){>}0\right){-}\pi\left((r^a,r^b): r_{1}^b(s){>}0\right)
    \right|\diff s \\ & \le \int_0^t\int_{\omega=(r^a,r^b)\in{\cal D}_T^2} |\ind{r_{1}^a(s){>}0}{-}\ind{r_{1}^b(s){>}0}| \pi(\diff \omega)\diff s\\
    &{\le}\int_0^t\int_{\omega=(r^a,r^b)\in{\cal D}_T^2} |r_{1}^a(s){-}r_{1}^b(s)|\land 1\,\pi(\diff \omega)\diff s.
\end{align*}
By taking the infimum among all the couplings of $\pi_a$ and $\pi_b$, we have
\[
    \int_0^t\left|\pi_{a}(r_1(s){>}0)- \pi_{b}(r_1(s){>}0)\right|\,\diff s\leq \int_0^t\rho_s(\pi_a,\pi_b)\diff s.\tag{e}
\]

Now, by combining  the estimates~(a),~(b),~(c),~(d), ~(e), we conclude
\begin{multline*}
    \Ept{\left(\|R_{\pi_a}- R_{\pi_b}\right\|_{\infty,t}}\\\leq
    (2\lambda+3\mu)\int_0^t\Ept{\left\|R_{\pi_a}-R_{\pi_b}\right\|_{\infty,s}}\,\diff s+\lambda\int_0^t\rho_s(\pi_a,\pi_b)\diff s,
\end{multline*}
Gr\"onwall's inequality then gives
\[
    \Ept{\left(\|R_{\pi_a}- R_{\pi_b}\right\|_{\infty,t}}\leq C_T\int_0^t\rho_s(\pi_a,\pi_b)\diff s, \quad\quad \forall t\in[0,T],
\]
with $C_T{=}\lambda \exp(2\lambda{+}3\mu)T$. Hence using ~\eqref{rho}, we have
\begin{equation}\label{uniq_temp}
    \rho_t(\Psi(\pi_a),\Psi(\pi_b))\leq C_T \int_0^t\rho_s(\pi_a,\pi_b)\diff s, \quad\quad \forall t\in[0,T].
\end{equation}

Uniqueness of the fixed point for the equation $\Psi(\pi){=}\pi$ follows immediately from \eqref{uniq_temp}. Also, a typical iterative argument proves the existence: pick any $\pi_0{\in} {\cal P}({\cal D}_T)$, and define the sequence $(\pi_n)$ inductively by $\pi_{n+1}{=}\Psi(\pi_n)$. It follows from Relation~\eqref{uniq_temp} that
\[
    W_T(\pi_{n+1},\pi_n)\leq \rho_T(\pi_{n+1},\pi_n)\leq \frac{(TC_T)^n}{n!} \int_0^T\rho_s(\pi_1,\pi_0)\diff s.
\]
The metric space $({\cal P}({\cal D}_T),W_T)$  is complete, and therefore the sequence $(\pi_n)$ converges. Since  $\Psi$ is continuous with respect to the Skorohod topology, its limit is necessarily a fixed point of $\Psi$. This completes the proof.
\end{proof}

\section{Mean-Field Limit} \label{MF-sec}
The empirical distribution $\Lambda^N(t)$ of $(R_{i}^N(t), 1{\leq} i{\leq} N)$ is defined by, for  $f$  a function on $\N^2$,
\[
    \Lambda^N(t)(f)=\frac{1}{N}\sum_{i=1}^N f(R_{i}^N(t))=\frac{1}{N}\sum_{i=1}^N f\left((R_{i,1}^N(t),R_{i,2}^N(t))\right).
\]
As it has already been remarked, at the beginning of Section~\ref{McV-sec}, the process  $(\Lambda^N(t))$ does not have the Markov property.  The goal of this section is to prove that the stochastic  process $(\Lambda^N(t))$ is converging in distribution as $N$ goes to infinity,  that is, for any function $f$ with finite support, the sequence of stochastic processes $(\Lambda^N(t)(f))$  converges in distribution. See Billingsley~\cite{Billingsley} and Dawson~\cite{Dawson}.

The main result of this section is the following theorem. 
\begin{theorem} (Mean-Field Convergence Theorem)\label{MF-Theo}
Suppose the process $(X^N(t))$ is initialized according to Assumption~\ref{asm_initial}. The sequence of empirical distribution process  $(\Lambda^N(t))$ converges in distribution to a process $(\Lambda(t)){\in}{\cal D}(\R_+,{\cal P}(\N^2))$ which is defined as follows:  for $f$ with finite support on $\N^2$,
\[
    \Lambda(t)(f)\stackrel{\text{def.}}{=} \Ept{f\left(\overline{R}_1(t),\overline{R}_2(t)\right)},
\]
where $(\overline{R}_1(t),\overline{R}_2(t))$ is  the unique solution of Equations~\eqref{McKean}.
Moreover, for any $p{\geq} 1$, the sequence of finite marginals $(R_{i,1}^N(t),R_{i,2}^N(t),1{\leq} i{\leq} p)$ converges in distribution to $((\overline{R}_{i,1}(t),\overline{R}_{i,2}(t)),1{\leq} i{\leq} p)$, where $(\overline{R}_{i,1}(t),\overline{R}_{i,2}(t))$ are  i.i.d.\ processes with the same distribution as $(\overline{R}_1(t),\overline{R}_2(t))$.
\end{theorem}
The last statement is the ``propagation of chaos'' property.

\subsection{Uniform Bound for $\mathbf{(R_i^N(t))}$}
We start with a technical result which will be used to establish mean-field convergence. It states that, uniformly on a compact time interval, the number of files with a copy at a given server $i$ is stochastically bounded  and that, with a high probability, all other servers have at most one file in common with server $i$. This is a key ingredient to prove that the non-Markovian process  $(R_{i,1}^N(t),(R_{i,2}^N(t))$ is converging in distribution to the nonlinear Markov process described in Theorem~\ref{McKeanTheo}.
\begin{lemma}\label{Tech_set}
If the initial state of the process $(X^N(t))$ is given by  Assumption~\ref{asm_initial}, for $1{\leq} i{\leq} N$ and $T{>}0$ then, for $i{\in}\N$,
\begin{equation}\label{ll1}
    \sup_{N\geq 1} \E\left(\sup_{0\leq t\leq T} \left(R_{i,1}^N(t)+R_{i,2}^N(t)\right)^2 \right)<+\infty
\end{equation}
and if
\[
    {\cal E_i}^N(T)\stackrel{\text{def.}}{=}\left\{\sup_{0\leq t\leq T,1\leq j\leq N} X_{i,j}^N(t)\geq 2\right\},
\]
then there exists a constant $C(T)$ independent of $i$ such that $\P\left({\cal E_i}^N(T)\right){\leq} C(T)/N$.
\end{lemma}

\begin{proof}
For $i{=}1,...,N,$  the total number of files $D_{i,0}^N$ initially on server $i$ satisfies
\begin{equation}\label{not1}
    D_{i,0}^N\stackrel{\text{def.}}{=}R_{i,1}^N(0){+}R_{i,2}^N(0)=X^N_{i,0}(0)+\sum_{j\not=i}^N X_{i,j}^N(0)= A_i{+}\sum_{j\not=i}^N \sum_{\ell=1}^{A_j} \ind{V_{j,\ell}^N=i},
\end{equation}
and hence, $\Ept{D_{i,0}^N}{=}2\Ept{A_1}$ and $\var(D_{i,0}^N){=} 2\var(A_1){+}\Ept{A_1}{N}/{(N{-}1)}$. Also, the total number of files $D_{i,1}^N(t)$ copied on server $i$ from all other servers during the interval $[0,t]$ verifies
\begin{equation}\label{not2}
    D_{i,1}^N(t)\stackrel{\text{def.}}{=}\sum_{j\not=i}^N \int_0^t \ind{R^N_{j,0}(s)>0}\Nun_{\lambda,j}(\diff s,\{i\})\leq \sum_{j\not=i}^N \Nun_{\lambda,j}(t,\{i\}).
\end{equation}
Therefore, for every $t{\leq} T$, $\Ept{D_{i,1}^N(t)}{\leq}\lambda T$ and $\Ept{D_{i,1}^N(t)^2}{\leq}2\lambda T.$ The bound \eqref{ll1} then follows from the inequality 
\[
\sup_{0\leq t\leq T}\left(R_{i,1}^N(t)+R_{i,2}^N(t)\right)\leq D_{i,0}^N(T)+D_{i,1}^N(T)
\]
For the next part, note that on ${\cal E_i}^N(T)$, there exists $1{\leq}j{\leq}N$ such that either server $i$ or $j$  makes two copies on the other one or both $i$ and $j$ make one copy on the other during the time interval $[0,T]$. Recall again that server $i$ initially copies $A_i$ files on other servers, and that the total number of files copied from server $i$ onto server $j$ during $(0,T]$ is upper bounded by $\Nun_{\lambda,i}(T,\{j\})$. Define the sequence  $(Z_{i,\ell}^N, 1{\leq} i{\leq} N, \ell{\geq}1)$ as follows: $Z_{i,\ell}^N{=}V^N_{i,\ell}$ when $1{\leq}\ell{\leq}A_i$, and $Z_{i,\ell}^N{=}U^i_{\ell-A_i}$ when $\ell{>}A_i$. For the first $A_i$ indices $\ell,$ $Z_{i,\ell}^N$s are therefore the indices of servers which received an initial copy of a file of server $i$, while the subsequent $Z_{i,\ell}^N$s are the server indices on which (potential) duplications from server $i$ can take place. $(Z_{i,\ell}^N)$ is therefore a sequence of i.i.d.\ random variables uniformly distributed on $\{1,\ldots,N\}{\setminus}\{i\}$. Therefore,  $\P\left({\cal E}_i^N(T)\right){\leq} \P\left({\cal B}_i^N\right)$, where
\begin{multline*}
    {\cal B}_i^N\stackrel{\text{def.}}{=}\bigcup_{j=1, j\not=i}^N \left( \bigcup_{\displaystyle\substack{1\leq \ell \leq A_i+L_i(T)\\ 1\leq \ell' \leq A_j+L_j(T)}}\hspace{-7mm}
    \{Z_{i,\ell}^N{=}j, Z_{j,\ell'}^N{=}i\}\right. \\ \ \hspace{-7mm}\left. \bigcup_{1\leq \ell\neq\ell' \leq A_i+L_i(T)} \hspace{-7mm}\{Z_{i,\ell}^N{=}j, Z_{i,\ell'}^N=j\}\hspace{-7mm}\bigcup_{1\leq \ell\neq\ell' \leq A_j+L_j(T)} \hspace{-7mm}\left\{V_{j,\ell}^N{=}i, V_{j,\ell'}^N{=}i\right\}\right),
\end{multline*}
with $L_i(T){=}{\cal N}_{\lambda,i}([0,T]){+}A_i$. Since the probability of  each of the elementary events of the right hand side of this relation is $1/(N{-}1)^2$, $Z_{k,\ell}$s are independent of $L_{k'}(T)$ for all $k,k'$, and $\Ept{L_i(T)}{=}\lambda T{+}\Ept{A_1}$. It is then easy to conclude.
\end{proof}

\subsection{Evolution Equations for the Empirical Distribution}
Denote $e_1{=}(1,0)$ and $e_2{=}(0,1)$, and define the operators
\begin{align*}
    \Delta^{\pm}(f)(x) &{=} f(x{+}e_1{-}e_2){-}f(x),\quad
    \Delta^{\mp}(f)(x) {=} f(x{-}e_1{+}e_2){-}f(x),\\
    \Delta^{+}_2(f)(x)& {=} f(x{+}e_2){-}f(x),
\end{align*}
for $x{\in}\N^2$ and $f{:}\N^2{\to}\R_+$. For every function $f{:}\N^2{\to} \R_+$ with finite support, it follows from Equations~\eqref{eqR1} and~\eqref{eqR2} and using martingale decomposition for the Poisson processes, we have
\begin{multline}\label{SDEi}
    \diff f(R_{i}^N(t))= \diff {\cal M}^N_{f,i}(t) + \Delta^{\mp}(f)(R_{i}^N(t))\ind{R_{i,1}^N(t)>0}\lambda\diff t \\
    {+}\Delta_2^+(f)(R_{i}^N(t))\frac{\lambda }{N{-}1} \sum_{j\not=i} \ind{R_{j,1}^N(t)>0}\diff t
    +  \left[f(0,0){-}f(R_{i}^N(t))\right]\mu\diff t \\
    + \sum_{j\neq i} \left[\rule{0mm}{4mm}f(R_{i}^N(t)+X_{i,j}^N(t)(e_1{-}e_2)){-}f(R_{i}^N(t))\right]\mu\diff t,
\end{multline}
where ${\cal M}^N_{f,i}$  is a martingale. The $j$th term of the last sum in on the right-hand side above corresponds to the event when server $j$ breaks down and therefore the copies of $X_{i,j}^N(t)$ files at node $j$ are lost, and the remaining copies are only located at node~$i$. Using the notation of Lemma~\ref{Tech_set} then, outside the event ${\cal E}_i^N(T)$, $X_{i,j}^N(t)$ is either 0 or 1, and hence, $t{\in}[0,T]$,
\[
    \sum_{j\neq i} \left[\rule{0mm}{4mm}f(R_{i}^N(t)+X_{i,j}^N(t)(e_1{-}e_2)){-}f(R_{i}^N(t))\right]= R_{2,i}^N(t)\Delta^{\pm}(f)(R_i^N(t)).
\]
By summing up both sides of Relation~\eqref{SDEi} over $i$ and denoting $\N^*{=}\N{\setminus}\{0\}$, we have
\begin{multline}\label{Emp1}
    \Lambda^N(t)(f){=} \Lambda^N(0)(f){+}{\cal M}^N_f(t) {+}  \lambda\int_0^t \int_{\N^2} \Delta^{\mp}(f)(x,y)\ind{x>0}\Lambda^N(s)(\diff x,\diff y)\diff s \\
    {+}\frac{\lambda N }{N{-}1}\int_0^t \Lambda^N(s)(\N^*{\times}\N) \int_{\N^2} \Delta_2^+(f)(x,y) \Lambda^N(s)(\diff x,\diff y)\diff s-H_1^N(t)\\
    + \mu \int_0^t \int_{\N^2} (f(0,0){-}f(x,y))\Lambda^N(s)(\diff x,\diff y)\diff s\\
    +\mu\int_0^t \int_{\N^2} y \Delta^{\pm}(f)(x,y)\,\Lambda^N(s)(\diff x,\diff y)\diff s +H_{2}^N(t),
\end{multline}
where 
\begin{align*}
    {\cal M}^N_f(t)&=\frac{1}{N}\left( {\cal M}^N_{f,1}(t)+{\cal M}^N_{f,2}(t)+\cdots+{\cal M}^N_{f,N}(t)\right),\\
    H_1^N(t)&=\frac{\lambda }{N{-}1}\int_0^t\int_{\N^2} \Delta_2^+(f)(x,y) \ind{x>0}\Lambda^N(s)(\diff x,\diff y)\diff s,\\
    H_2^N(t)&=\mu \frac{1}{N}\sum_{i=1}^N\int_0^t h_{2,i}^N(s)\,\diff s,
\end{align*}
with
\begin{multline*}
    h_{2,i}^N(t)=\sum_{j\neq i} \left(\rule{0mm}{4mm}f(R_{i}^N(t)+X_{ij}^N(t)(e_1{-}e_2)){-}f(R_{i}^N(t))\right)\\-\int_{\N^2} y \Delta^{\pm}(f)(x,y)\,\Lambda^N(t)(\diff x,\diff y).
\end{multline*}
Now, we investigate the asymptotic properties of the terms of the right hand side of  Equation~\eqref{Emp1}.

\subsection{The negligible terms} We first prove that the two processes $(H_1^N(t))$ and $(H_2^N(t))$  converge to zero in distribution as $N$ goes to infinity. For the former, the result follows immediately from the simple bound \[\|H_1^N\|_{\infty,T}\leq \frac{2\lambda T}{N}\|f\|_\infty.\]

For $(H_2^N(t))$, first note that, for $0{\leq} t{\leq} T$ and $1{\leq} i{\leq} N$, $h_{2,i}^N(t)$ is non-zero only on the event $\cal{E}^N_i(T)$, and hence,
\begin{align*}
   \left\|\frac{1}{N}\sum_{i=1}^Nh_{2,i}^N(s)\right\|_{\infty,T} \leq  \sup_{0\leq s\leq T}\frac{1}{N}\sum_{i=1}^N\left|\sum_{j\neq i} \left[ f(R_{i}^N(s){+}X_{ij}^N(s)(e_1{-}e_2)){-}f(R_{i}^N(s))\right] \right.\\\left. \rule{0mm}{6mm}{-}R_{2,i}^N(s)\Delta^{\pm}(f)(R_i^N(s))\right|    \leq 4\mu\|f\|_\infty \frac{1}{N}\sum_{i=1}^N|R_{2,i}^N(s)|_{\infty,T}\ind{\cal{E}^N_i(T)}.
\end{align*}
By an application of Cauchy-Schwartz inequality and using Lemma~\ref{Tech_set}, there exists a constant $C_1(T)$ such that
\[
    \frac{1}{N}\sum_{i=1}^N\Ept{|R_{2,i}^N|_{\infty,T} \ind{\cal{E}^N_i(T)}} \leq \frac{1}{N}\sum_{i=1}^N\sqrt{\Ept{|R_{2,i}^N|_{\infty,T}^2}}\sqrt{\P\left(\cal{E}^N_i(T)\right)}
     {\le}\frac{C_1(T)}{\sqrt{N}}.
\]
Consequently,
\[
    \lim_{N\to+\infty} \Ept{\sup_{0\leq t\leq T} \left|\frac{1}{N}\sum_{i=1}^N\int_0^t h_{2,i}^N(s)\,\diff s \right|}=0,
\]
which implies that the process $(H_2^N(t))$ is also vanishing in distribution.

\subsection{The Martingale}
Careful calculations show that the previsible increasing process of $({\cal M}^N_f(t))$ is given by
\[
    \left(\croc{{\cal M}^N_f}(t)\right){=} \left(\frac{\lambda}{N^2}G_1^N(t)+\frac{\mu}{N^2}G_2^N(t)\right),
\]
with 
\begin{multline*}
    G_1^N(t)=\sum_{i=1}^N  \int_0^t
    \left(\rule{0mm}{5mm} \Delta^{\mp}(f)(R_{i}^N(s)) {+}\frac{N}{N{-}1}\Lambda^N(s)(\Delta_2^{+}(f))\right.\\ \left. \rule{0mm}{5mm}{-}\frac{N}{N{-}1}\Delta_2^{+}(f)(R_i^N(t))\right)^2\ind{R_{i,1}^N(s-)>0}\,\diff s,
\end{multline*}
and
\begin{multline*}
    G_2^N(t)=\sum_{i=1}^N  \int_0^t\left(\rule{0mm}{5mm}
    f(0,0)-f(R_i^N(s))\right.  \\ \left.  \rule{0mm}{5mm}{+}\sum_{j\not=i} [f(R^N_j(s){+}X_{i,j}^N(s)(e_1{-}e_2)){-}f(R^N_j(s))]\ind{X_{ij}^N(s)>0}\right)^2\,\diff s.
\end{multline*}
From the simple bounds $\|G_1^N\|_{\infty,T} {\leq}  16\cdot NT\|f\|_{\infty}^2$  and
\[
\|G_2^N\|_{\infty,T}  \leq 8NT\|f\|_{\infty}^2\left(1{+} \|R_{i,2}^N\|_{\infty,T}^2\right),
\]
and, by using Relation~\eqref{ll1} of Lemma~\ref{Tech_set}, we get 
\[
    \lim_{N\to+\infty} \Ept{\croc{{\cal M}^N_f}(T)} =0.
\]
Therefore, by Doob's inequality, the martingale $(\cal{M}_f^N(t))$ converges to zero in distribution as $N$ goes to infinity.

\begin{prop}\label{tightprop}
(Tightness of the Empirical Distribution Process) The sequence  $(\Lambda^N(t))$  is tight with respect to the convergence in distribution in ${\cal D}(\R_+,{\cal P}(\N^2))$. Any limiting point $(\Lambda(t))$  is a continuous process which satisfies
\begin{multline}\label{WeakLimit}
    \Lambda(t)(f){=}\Lambda(0)(f){+}\lambda \int_0^t \int_{\N^2} \Delta^{\mp}(f)(x,y)\ind{x>0}\Lambda(s)(\diff x,\diff y)\diff s \\
    {+}\lambda\int_0^t \Lambda(s)(\N^*{\times}\N) \int_{\N^2} \Delta_2^+(f)(x,y) \Lambda(s)(\diff x,\diff y)\diff s\\
    + \mu \int_0^t \int_{\N^2} (f(0,0)-f(x,y))\Lambda(s)(\diff x,\diff y)\diff s\\
    +\mu\int_0^t \int_{\N^2} y \Delta^{\pm}(f)(x,y)\,\Lambda(s)(\diff x,\diff y)\diff s
\end{multline}
for every function $f$ with finite support on $\N^2$.
\end{prop}

Note that the Fokker-Planck Equation~\eqref{FK} of the introduction  is the functional form of the stochastic Equation~\eqref{WeakLimit}.

\begin{proof}
Theorem~3.7.1 of Dawson~\cite{Dawson} states that it is enough to prove that, for any   function $f$ on $\N^2$ with finite support,  the sequence of processes $(\Lambda_N(\cdot)(f))$ is tight with respect to the topology of the uniform norm on compact sets. Using the criterion of the modulus of continuity (see e.g.\ Theorem 7.2, page~81 of Billingsley~\cite{Billingsley}), we need to show that for every $\eps{>}0$ and $\eta{>}0$, there exists a $\delta_0{>}0$ such that if $\delta{<}\delta_0$ then, 
\begin{equation}\label{tight_cond}
    \P\left(\sup_{\substack{0\leq s\leq t\leq T\\ |t-s|\leq \delta}}\left|\Lambda_N(t)(f)-\Lambda_N(s)(f)\right|\geq \eta\right)\leq \eps
\end{equation}
holds for all $N{\in}\N$. Fix $0{\leq} s, t{\leq} T$ with $|t{-}s|{\leq}\delta$, and remember the equality~\eqref{Emp1} for the process  $(\Lambda^N(t)(f))$. We have already shown that  the processes $(H_1^N(t))$, $(H_2^N(t))$, and $({\cal M}_f^N(t))$ vanish as $N$ goes to infinity. For the remaining terms on the right-hand side of \eqref{Emp1}, note that there exists a finite  constant $C_0$ such that 
\[\left|\int_s^t \int_{\N^2} \Delta^{\mp}(f)(x,y)\ind{x>0}\Lambda^N(u)(\diff x,\diff y)\diff u\right|\leq C_0\delta \|f\|_{\infty},\]
\[\left|\int_s^t \Lambda^N(u)(\N^*{\times}\N) \int_{\N^2} \Delta_2^+(f)(x,y) \Lambda^N(u)(\diff x,\diff y)\diff u\right|\leq C_0\delta \|f\|_{\infty},\]
and
\[\left|\int_s^t \int_{\N^2} (f(0,0)-f(x,y))\Lambda^N(u)(\diff x,\diff y)\diff u\right|\leq C_0\delta \|f\|_{\infty}.\]
Also, by Relation~\eqref{ll1} of Lemma~\ref{Tech_set} shows that there exists $C_1{<}\infty$ independent of $N$ such that 
\begin{multline*}
  \Ept{\left|\int_s^t \int_{\N^2} y \Delta^{\pm}(f)(x,y)\,\Lambda^N(u)(\diff x,\diff y)\diff u\right|}\\ \leq 2\|f\|_{\infty}\delta\frac{1}{N}\sum_{i=1}^N \Ept{\sup_{0\leq u\leq T} R_{i,2}^N(u)}\leq C_1\delta\|f\|_\infty.
\end{multline*}
If follows from the Chebishev's inequality that the sequence $(\Lambda^N(\cdot)(f))$ satisfies Relation~\eqref{tight_cond}, and hence it is tight.  
 
Moreover, if $\Lambda$ is a limiting point, from Relation~\eqref{Emp1}  and the fact that  the processes $H_1^N$, $H_2^N$, and ${\cal M}_f^N$ vanish as $N$ gets large, one obtains that Relation~\eqref{WeakLimit} holds,
Finally, it is straightforward to show that all the terms on the right-hand side of \eqref{WeakLimit} are continuous in $t$.
\end{proof}
We now show that  Equation~\eqref{WeakLimit} that characterized the  limits of $(\Lambda^N(t))$ has a unique solution. 
\begin{lemma}\label{lem_bound}
Let $(\Lambda(t))$ be a solution to equation \eqref{WeakLimit} with an initial condition $\Lambda(0)$, a probability on $\N^2$ with bounded support.  Then, for any $T{>}0$, there exists a constant $C_T$ such that for all $K{\geq}2\log(2)$,
\begin{equation}\label{claim}
    \sup_{0\leq t\leq T}\int_0^t\int_{\N^2}y\ind{y\ge K}\Lambda(s)(\diff x,\diff y)\diff s \le C_Te^{-K/2}.
\end{equation}  
\end{lemma}
\begin{proof}
For all $t\leq T$, since $y\leq \exp(y/2)$ if $y\geq 2\log(2)$, then for $K\geq 2\log(2)$,
\begin{equation}\label{claim2}
    \int_{\N^2}y\ind{y\ge K}\Lambda(t)(\diff x,\diff y)\leq e^{-K/2}\int_{\N^2}e^{y}\Lambda(t)(\diff x,\diff y).
\end{equation}
For every $K_1\geq0,$ using equation \eqref{WeakLimit} for $\Lambda$ with $f$ replaced by
\[
\tilde f(x,y)=e^{x+y}\ind{x+y\leq K_1},
\]
and since 
$\Delta^{\mp}(\tilde f)=\Delta^{\pm}(\tilde f)=0,$
we have
\begin{multline*}
    \int_{\N^2}e^{x+y}\ind{x+y\le K_1}\Lambda(t)(\diff x,\diff y){\le}
    \int_{\N^2}e^{x+y}\Lambda(0)(\diff x,\diff y)\\
    {+}\lambda (e-1)\int_0^t\int_{\N^2}e^{x+y}\ind{x+y\le K_1}\Lambda(s)(\diff x,\diff y)\diff s\\
    {+}\mu\int_0^t\left(1-\int_{\N^2}e^{x+y}\ind{x+y\le K_1}\Lambda(s)(\diff x,\diff y)\right)\diff s.
\end{multline*}
By an application of Gr\"onwall's inequality, there exists a constant $c_T$ independent of $K_1$ such that
\[
    \sup_{0\leq t\le T}\int_{\N^2}e^{y}\ind{x+y\le K_1}\Lambda(t)(\diff x,\diff y) \le c_T.
\]
The bound ~\eqref{claim} can be obtained by letting $K_1$ go to infinity in the above inequality, and substituting it in Relation~\eqref{claim2}.
\end{proof}

\begin{prop}[Uniqueness]\label{UniqProp}
For every $\Lambda_0$ a probability on $\N^2$ with finite support, Equation~\eqref{WeakLimit} has at most one solution $(\Lambda(t))$ in ${\cal D}(\R^+,{\cal P}(\N^2))$, with  initial condition $\Lambda_0$.
\end{prop}
\begin{proof}
Let $(\Lambda^{1}(t))$ and $(\Lambda^{2}(t)){\in} {\cal D}(\R_+,{\cal P}(\N^2))$  be solutions of ~\eqref{WeakLimit} with initial condition $\Lambda_0$. Let $f$ be a  bounded function on $\N^2$ and $t{\geq} 0$, we have
\begin{align*}
    \Lambda^1(t)(f)&-\Lambda^2(t)(f){=}\lambda \int_0^t \int_{\N^2} \Delta^{\mp}(f)(x,y)\ind{x>0}\left(\Lambda^1(s)-\Lambda^2(s)\right)(\diff x,\diff y)\diff s \\
    &{+}\lambda\int_0^t \Lambda^1(s)(\N^*{\times}\N) \int_{\N^2} \Delta_2^+(f)(x,y) \left(\Lambda^1(s)-\Lambda^2(s)\right)(\diff x,\diff y)\diff s\\
    &{+}\lambda\int_0^t \left(\Lambda^1(s)-\Lambda^2(s)\right)(\N^*{\times}\N) \int_{\N^2} \Delta_2^+(f)(x,y) \Lambda^2(s)(\diff x,\diff y)\diff s\\
    &{+} \mu \int_0^t \int_{\N^2} (f(0,0)-f(x,y))\left(\Lambda^1(s)-\Lambda^2(s)\right)(\diff x,\diff y)\diff s\\
    &{+}\mu\int_0^t \int_{\N^2} y\Delta^{\pm}(f)(x,y)\,\left(\Lambda^1(s)-\Lambda^2(s)\right)(\diff x,\diff y)\diff s.
\end{align*}
For a signed measure $m$ on $\N^2$, denote  
\[
    \|m\|_{TV}=\sup\left\{\int_{\N^2}f(x,y)m(\diff x,\diff y), f:\N^2\to \R \textrm{ with }\|f\|_\infty\le 1\right\}.
\]
Therefore, for every $f$ on $\N^2$ with $\|f\|_\infty{\leq}1$ and every $K{>}0$, we have
\begin{multline*}
    \left|\Lambda^1(t)(f)-\Lambda^2(t)(f)\right|{\le}(6\lambda{+}2\mu{+}2\mu K)\int_0^t\|\Lambda^1(s)-\Lambda^2(s)\|_{TV}\diff s\\
    {+}2\mu\int_0^t\int_{\N^2}y\ind{y\ge K}\Lambda^1(s)(\diff x,\diff y)\diff s{+}2\int_0^t\int_{\N^2}y\ind{y\ge K}\Lambda^2(s)(\diff x,\diff y)\diff s.
\end{multline*}
Now using \eqref{claim} of Lemma \ref{lem_bound}, and taking the supremum over all functions $f$ on $\N^2$ with $\|f\|_\infty\leq1$, we have
\[
\|\Lambda^1(t)-\Lambda^2(t)\|_{TV} \leq 4\mu C_te^{-K/2} +(6\lambda+2\mu+2\mu K)\int_0^t\|\Lambda^1(s)-\Lambda^2(s)\|_{TV}\diff s.
\]
Therefore, by another application of  Gr\"onwall's inequality,
\[
    \|\Lambda^1(t)-\Lambda^2(t)\|_{TV} \le 4\mu C_T e^{-K/2}e^{(6\lambda+2\mu+2\mu K)t}.
\]
For $t{<}1/(4\mu)$, by letting $K$ go to infinity in the above relation, one gets that $\Lambda^1(t){=}\Lambda^2(t)$. By repeating the same argument on successive time intervals  of width less than $1/(4\mu)$, one obtains the uniqueness result. 
\end{proof}
Now we can conclude the proof of Theorem \ref{MF-Theo}.
\begin{proof}[Proof of Theorem \ref{MF-Theo}]
Let $(x,y){\in} N^2$ and $\Lambda_0{=}\delta_{(x,y)}$, then  if  $(\overline{R}_1(t),\overline{R}_2(t))$ is the unique solution of Equation~\eqref{McKean} and  the measure valued process $(\Lambda^1(t))$ is defined  by,  if $f$ is a function with finite support on $\N^2$,
\[
    \Lambda^1(t)(f)\stackrel{\text{def.}}{=} \Ept{f(\overline{R}_1(t),\overline{R}_2(t))},
\]
it is straightforward to check that this is a solution of Equation~\eqref{WeakLimit}.  The convergence of $(\Lambda^N(t))$ follows from Propositions~\ref{tightprop} and~\ref{UniqProp}. The last assertion is a simple consequence of Proposition~2.2 in Sznitman~\cite{Sznitman}.
\end{proof}

\section{An Asymptotic Bound on the Decay of the Network}\label{Decay-Sec}
The asymptotic  process $(\overline{R}(t)){=}(\overline{R}_1(t),\overline{R}_2(t))$   of Theorem~\ref{McKeanTheo}  is an inhomogeneous Markov process with the following transitions: if $(\overline R(t))$ is in state $r{=}(r_1,r_2)$ at time $t$, the next possible state and the corresponding rates are given by
\begin{equation}\label{eqdd}
    r\mapsto \begin{cases} (0,0) &\text{with rate } \mu\\ r{+}e_2&\text{with rate } \lambda p(t) \end{cases}
    \text{ and} \quad
    r\mapsto \begin{cases}   r{-}e_1{+}e_2 &\text{with rate }\lambda\ind{r_1>0}\\   r{+}e_1{-}e_2 &\text{with rate } \mu r_2, \end{cases}
\end{equation}
where $p(t){=}\P\left(\overline{R}_1(t){>}0\right)$ is the non-linear part of the dynamic. A simple feature of this process is that it resets to the state $(0,0)$ at the epoch times of a Poisson process with rate $\mu$, and between two consecutive epoch times, the sum of its coordinates grows according to an inhomogeneous Poisson process with rate $p(\cdot)$. With this observation, the following proposition gives a representation of the distribution of the total number of copies with the function $(p(t))$. 
\begin{prop}\label{gen-prop}
If the initial state of $(\overline{R}_1(t),\overline{R}_2(t))$ is $(0,r_2)$ with $r_2{\in}\N$ then, for $u{\in}[0,1]$ and $t{\geq} 0$,
\begin{multline*}
    \Ept{u^{\overline{R}_1(t)+\overline{R}_2(t)}}=e^{-\mu t} u^{r_2} \exp\left(-\lambda(1{-}u) \int_0^tp(z)\,\diff z\right)\\+\int_0^{t} \exp\left(-\lambda(1{-}u) \int_{0}^sp(t{-}z)\,\diff z\right) \mu e^{-\mu s}\,\diff s.
\end{multline*}
\end{prop}
\begin{proof}
From Relation~\eqref{eqdd},   one obtains that the transition rates of the process $(\overline{R}_1(t){+}\overline{R}_2(t))$ are given by
  \[
r\mapsto   r{+}1, \text{ at rate }\lambda p(t) \text{ and }     r\mapsto     0,  \text{ at  rate }  \mu.
\]
The Fokker-Planck equation  associated to this process yields the relation
\[
\frac{\diff}{\diff t} \Ept{u^{\overline{R}_1(t)+\overline{R}_2(t)}}=\mu+\left( \lambda p(t) (u{-}1){-}\mu\right)\Ept{u^{\overline{R}_1(t)+\overline{R}_2(t)}}.
\]
It is then easy to conclude. 
\end{proof}
The problem with the above formula is that the function $t\mapsto p(t)$ is unknown. In the following, we obtain a lower bound on the asymptotic rate of decay of the network, i.e.\ the exponential rate of convergence of the process $(\overline{R}_1(t),\overline{R}_2(t))$ to $(0,0)$.

Recall that $R_{i,1}^N(t)$ and $R_{i,2}^N(t)$ are the number of files on server $i$ with one and two copies, respectively. Therefore, the quantity
\[
    L^N(t)=\sum_{i=1}^N R_{i,1}^N(t)+\frac{1}{2}\sum_{i=1}^N R_{i,2}^N(t)
\]
is the total number of distinct files in the system at time $t$. By Theorem~\ref{MF-Theo}, Equation~\eqref{ll1} of Lemma~\ref{Tech_set}, and an application of the dominated convergence theorem, we have
\[
   (L(t))\steq{def}\lim_{N\to\infty}\left(\frac{L^N(t)}{N}\right)=\left(\Ept{\overline{R}_1(t)}+ \frac{1}{2}\Ept{\overline{R}_2(t)}\right).
\]
The following proposition gives therefore a lower bound on the exponential rate of decay $(L(t)).$
\begin{prop}\label{up-prop}
If $(\overline{R}(t)){=}(\overline{R}_1(t),\overline{R}_2(t))$ is the solution of Equation~\eqref{McKean}, then
\begin{equation}\label{Up}
\Ept{\overline{R}_1(t) +\frac{1}{2} \overline{R}_2(t)}\le \Ept{\overline{R}_1(0) +\frac{1}{2} \overline{R}_2(0)} e^{-\kappa_2^+(\rho) \mu t},
\end{equation}
where  
\[
\kappa_2^+(x)=\frac{(3{+}x)-\sqrt{(3{+}x)^2{-}8}}{2}, \qquad x\in\R,
\]
and $\rho{=}\lambda/\mu$. 
\end{prop}

The quantity $\kappa_2^+(\rho) \mu$ is thus a lower bound for the exponential rate of decay. When  there is no duplication capacity, i.e.\  $\lambda{=}0$, $\kappa_2^+(\rho){=}1$ and the lower bound becomes $\mu$, the failure rate of servers, as expected. On the other hand, when the duplication capacity goes to infinity, the lower bound goes to $0$.
\begin{proof}[Proof of Proposition \ref{up-prop}]
Let $m_1(t){=}\Ept{\overline{R}_1(t)}$ and $m_2(t){=}\Ept{\overline{R}_2(t)}$. Taking expectation from both sides of Equations~\eqref{McKean}, we conclude that the pair $(m_1,m_2)$ satisfy the following set of Ordinary Differential Equations (ODEs):
\begin{equation}\label{ode}
\begin{cases}
\dot m_1(t)=-\lambda p(t)+\mu(m_2(t)-m_1(t)),\\
\dot m_2(t)=2\lambda p(t)-2\mu m_2(t).
\end{cases}
\end{equation}
Defining $g(t){=}m_1(t){-}p(t)$, then clearly $0{\le} g(t){\le} m_1(t)$. The  ODEs~\eqref{ode} can be rewritten as
\begin{align}\label{ODE_vector}
\frac{\diff}{\diff t} {m_1(t)\choose m_2(t)}=A{m_1(t)\choose m_2(t)}+\lambda{g(t)\choose -2 g(t)},
\end{align}
where $A$ is the matrix
\[
A=\left(\begin{matrix}
{-}(\lambda{+}\mu)&\mu\\
2\lambda&{-}2\mu
\end{matrix}
\right).
\]
It  has two negative eigenvalues, ${-}\mu\kappa_2^+(\rho)$ and ${-}\mu\kappa_2^-(\rho)$ with
\[
\kappa_2^+(\rho)=\frac{(3{+}\rho)-\sqrt{(3{+}\rho)^2-8}}{2},\quad \kappa_2^-(\rho)=\frac{(3{+}\rho)+\sqrt{(3{+}\rho)^2{-}8}}{2}.
\]
Defining the constants $y_1{=}({-}\mu\kappa_2^+(\rho){+}\lambda{+}\mu)/\mu$, $y_2{=}({-}\mu\kappa_2^-(\rho){+}\lambda{+}\mu)/\mu$,
\[
h_1=\frac{1}{y_1{-}y_2}({-}y_2 m_1(0){+}m_2(0)), \quad\text{and}\quad h_2=\frac{1}{y_1{-}y_2}(y_1 m_1(0){-}m_2(0)),
\]
the standard formula for explicit solution of the linear ODE \eqref{ODE_vector}, with $g$ regarded as an external force,  gives
\begin{align}
  m_1(t)&=h_1e^{-\mu\kappa_2^+(\rho)t}+h_2e^{-\mu\kappa_2^-(\rho)t}\label{m1}\\&-\frac{\lambda}{y_1{-}y_2}\int_0^tg(s)\left[(y_2{+}2)  e^{-\mu\kappa_2^+(\rho)(t{-}s)}{-}(y_1{+}2)e^{-\mu\kappa_2^-(\rho)(t-s)}\right]\,\diff s, \notag\\
m_2(t)&=y_1h_1e^{-\mu\kappa_2^+(\rho)t}+y_2h_2e^{-\mu\kappa_2^-(\rho)t}\label{m2} \\ & -\frac{\lambda}{y_1{-}y_2}\int_0^tg(s)\left[(y_2{+}2)y_1 e^{-\mu\kappa_2^+(\rho)(t-s)}{-}(y_1{+}2)y_2e^{-\mu\kappa_2^-(\rho)(t-s)}\right]\,\diff s.\notag
\end{align}
Therefore, using the fact $g(s){\geq}0$ in the first inequality, and the relations  $y_1{>}y_2{\geq}{-}2$ and $\kappa_2^+(\rho) {<}\kappa_2^-(\rho)$ in the second inequality below, we conclude
\begin{align*}
m_1(t)+\frac{1}{2}m_2(t)&=\left(1+\frac{y_1}{2}\right)h_1e^{-\mu\kappa_2^+ (\rho)t}+\left(1+\frac{y_2}{2}\right)h_2e^{-\mu\kappa_2^-(\rho)t}\notag \\
&-\frac{\lambda}{y_1{-}y_2}\frac{1}{2}(y_1{+}2)(y_2{+}2)\int_0^tg(s) \left(e^{-\mu\kappa_2^+(\rho)(t-s)}{-} e^{-\mu\kappa_2^-(\rho)(t-s)}\right)\diff s\notag\\
&\le\left(1+\frac{y_1}{2}\right)h_1e^{-\mu\kappa_2^+(\rho)t}+\left(1+\frac{y_2}{2}\right)h_2 e^{-\mu\kappa_2^-(\rho)t}\notag\\
&\leq \left(m_1(0)+\frac{1}{2}m_2(0)\right)e^{-\mu\kappa_2^+(\rho)t},\notag
\end{align*}
 This completes the proof.
\end{proof}
\section{The Case of Multiple Copies}\label{Ext-Sec}
In this section, we consider the general case where each file has a maximum number of $d$ copies in the system. We now describe the algorithm without too much formalism for sake of simplicity.    The duplication capacity of a given node is used for one of its copies corresponding to a file with the least number of copies in the network. Provided that this number is strictly less than $d$,  a new copy is done at rate $\lambda$ at  random   on a node of the network.  See Section~3 of Sun at al.~\cite{SSMRS} for a quick description of how this kind of mechanism can be implemented in practice. The node receives copies from other nodes  from this duplication mechanism. As before, at rate $\mu$ all copies of the node are removed.

As it will be seen, the model does not seem to be mathematically tractable as in the case $d{=}2$.  To understand the effect of the maximum number of copies $d$ on the performance of the file system, we study the  asymptotic behavior  of  a stochastic model which is dominating  (in some sense) our network.  We then study the decay rate of this new model.

The initial condition of our system  are given by the following assumption. 
\begin{assumption}\label{asm_initial2}
(Initial State)  There is a set ${\cal F}_N$ of $F_N$ initial files, and for $f{\in}{\cal F}_N$, a subset of $d$ nodes of $\{1,\ldots,N\}$ is taken at random and on each of them a copy of $f$ is done. 
\end{assumption}
It should be noted that, with the duplication mechanism described above, a copy can be made on a node which has already a copy of the same file. But, with  similar methods as the ones used  in the proof of Lemma~\ref{Tech_set}, it can be shown that on any finite time interval, with probability $1$, there is only a finite number of files which have at least two copies on a server.  In particular, this assumption has no influence on the asymptotic results obtained in this section since they are concerning asymptotic growth in $N$ of the number of files alive at time~$t$. 

With these assumptions, if $f$ is a file, $f{\in}{\cal F}_N$, one denotes by ${\cal A}_f^N(t){\subset}\{1,\ldots,N\}$  the subset of nodes which have a copy of $f$ at time $t$. The cardinality   of the set ${\cal A}_f^N(t)$  is denoted as $c^N_f(t)$, it is at most $d$. The process $$({\cal A}^N(t)){\steq{def}} ({\cal A}_f^N(t), f{\in}{\cal F}_N)$$ gives a (Markovian) representation of the time evolution of the state of the network. 
\subsection{The Additional Complexity of the Model}
A analogous  Markovian description as for the case $d{=}2$ can be done in the following way.  If ${\cal S}$ is the set of non-empty subsets  of $\{1,\ldots,N\}$ whose cardinality is less or equal to $d$ and, for $A\in {\cal S}$, if $X_A^N(t)$ is the number of files with a copy  only in the nodes whose index is in $A$,
\[
X_A^N(t)=\sum_{f{\in}{\cal F}_N} \ind{{\cal A}_f^N(t)=A}
\]
then it is not difficult to show that $(X^N(t)){=}(X_A^N(t),A\in{\cal S})$ is a Markov process, even if its transitions are not so easy to write formally. 
Following the analysis done for the case $d{=}2$, it is natural to introduce, for $1{\leq}i{\leq}N$, and $1{\leq}k{\leq}d$, the quantity
\[
R_{i,k}^N(t)=\sum_{\substack{A\in{\cal S},  i\in A\\ \card(A)=k}} X_A^N(t)=\sum_{f{\in}{\cal F}_N} \ind{i\in{\cal A}_f^N(t), c^N_f(t)=k}
\]
is the number of files having $k$ copies in the whole network, with a copy on server $i$. It is the equivalent of the variables $R_{i,1}^N(t)$ and $R_{i,2}^N(t)$ of the case $d{=}2$.

The vector $R_i^N(t){=}(R_{i,k}^N(t), 1{\leq} k{\leq} d)$ gives also a reduced representation of the state of the node $i$ at time $t$.  It turns out that the evolution equations of this model are much more involved. To observe why our method cannot be worked out for general $d{>}2$, let us try, as in Section~\ref{McV-sec}, to heuristically obtain the transition rates of a possible  asymptotic limit process for this model. 

Fix $1{\leq} k{<}d$ and $1 {\leq} i{\leq} N$, $e_k$ is the $k$th unit vector of $\N^d$, and $r{=}(r_j){=}(R_{i,j}^N(t-))$,  then  the process $(R_{i,j}^N(t)$ jumps   from $r$ to $r{-}e_k{+}e_{k+1}$ at time $t$ according with two types of events:
\renewcommand{\theenumi}{\alph{enumi}}
\begin{enumerate}
\item   due to the duplication capacity at node $i$, it occurs  at rate $\lambda$ under the condition  $r_1{=}r_2{=}\cdots{=}r_{k-1}{=}0$ and $r_k{>}0$. Recall that only files with the least number of copies are duplicated. 
  \item  If a copy of a file is present at $i$  with a total $k$ copies is duplicated on one of the other $k{-}1$ servers having a copy of this file,  conditionally  on the past before $t$, it occurs at rate 
 \[
    \lambda \sum_{j\not=i} \ind{R_{j,\ell}^N(t-){=}0, 1{\leq}\ell{<}k,R_{j,k}^N(t-)>0}\frac{1}{R_{j,k}^N(t-)}\sum_{\substack{A\in{\cal S}:  i,j\in A\\ \card(A)=k}}X_A^N(t-).
    \]
\end{enumerate}
The first event is similar as in  the case $d{=}2$, the jump rate can be expressed in terms of the vector $r$.  This is not the case for the second event. The last sum of the above expression does not seem to have an expression in terms of the components of the vector $r$. It requires a much more detailed description. The information provided by $r$ is not enough, even in the limit when $N$ goes to infinity as it is the case when $d{=}2$. Consequently,  it does not seem that one can derive autonomous equations describing the  asymptotic dynamics of $(R^N_i(t))$.

\subsection{Introduction of  a  Dominating Process}
We now consider the following related  Markov process. We first describe it, without too much formalism for sake of simplicity, in terms of a duplication system. Note that this is {\em not} an alternative algorithm but merely a way of having a mathematically tractable stochastic process that will give us a lower bound of the exponential  decay rate  of the initial system. For convenience,  we will use nevertheless the terminology of ``files'', ``copies'', ``duplication'' and ``servers'' to describe this new process.

The initial condition is also  given by Assumption~\ref{asm_initial2}. If $f$ is one of the files of the network, as long as the total number of copies of $f$ is strictly less that $d$, each of the nodes having one of these copies generates a new copy of $f$ at rate $\lambda$ at random on a node. The case of multiple copies on the same server is taken care of as for the process $({\cal A}^N(t))$.  The failures of a given node occur according to a Poisson process with rate $\mu$ and,  as for our algorithm,  all copies are lost.

The  system works the same as the original model, except that each file on a server $i$ with strictly less than $d$ copies in the network, can be duplicated at rate $\lambda$, regardless of any other copy on that server.
Consequently, if, for $1{\leq}k{\leq}d$, a node has $r_k$ files, each of them  with a total of $k$ copies,  at a given time, the ``duplication capacity'' of this node  is  given by 
\[
    \lambda \left( r_1{+}r_2{+}\cdots{+}r_{d-1}\right),
\]
instead of $\lambda$ in our algorithm. Remember nevertheless that such system  is not possible in practice, it is used only to estimate the performances of the algorithm introduced at the beginning of this section. 

For $f{\in}{\cal F}_N$, one denotes by ${\cal B}_f^N(t)$  the finite subset of $\{1,\ldots,N\}$ of nodes having a copy of $f$ at time $t$ and its  cardinality is denoted by   $d^N_f(t)$. We define
 \[
 ({\cal B}^N(t)){\steq{def}} ({\cal B}_f^N(t), f{\in}{\cal F}_N).
 \]
 Since the duplication mechanism associated to the process $({\cal B}^N_f(t))$ is more active than for our algorithm, intuitively the decay rate of our system should be faster that the decay rate of the process $({\cal B}^N_f(t))$. The following lemma will be used to establish rigorously this relation. 
\begin{lemma}
  There exists a coupling of the processes $({\cal A}^N(t))$ and  $({\cal B}^N(t))$  such that, almost surely,  for all $f{\in}{\cal F}_N$ and $t{\geq}0$, ${\cal A}^N_f(t){\subset} {\cal B}^N_f(t)$.
\end{lemma}
\begin{proof}
  This is done by induction on the number of jumps of the process  $({\cal B}^N(t))$. By assumption one can take   ${\cal A}^N(0){=}{\cal B}^N(0)$. If the relation  ${\cal A}^N_f(t){\subset} {\cal B}^N_f(t)$, at the instant $t=\tau_n$ of the $n$th jump of  $({\cal A}^N(t))$ and $({\cal B}^N(t))$. We review the different scenarios for the next jump after time $\tau_n$, at time $\tau_{n+1}$,
  \begin{enumerate}
  \item if some node $i_0{\in}\{1,\ldots,N\}$ fails, then, for $f{\in}{\cal F}_N$,
    \[
            {\cal A}^N_f(\tau_{n+1}){=}{\cal A}^N_f(\tau_{n}){\setminus}\{i_0\}\text{ if }   i_0{\in}{\cal A}^N_f(\tau_{n}), \qquad
      {\cal A}^N_f(\tau_{n+1}){=}{\cal A}^N_f(\tau_{n})  \text{ otherwise, }
      \]
      and a similar relation holds for ${\cal B}^N_f(\tau_{n+1})$. The relation ${\cal A}^N_f(t){\subset} {\cal B}^N_f(t)$ still holds for $t{=}\tau_{n+1}$ since it is true at time $\tau_n$. 
    \item If a duplication occurs for the process $({\cal A}^N(t))$ at time $\tau_{n+1}$ at some node  $i_0{\in}\{1,\ldots,N\}$ and for file a $f{\in}{\cal F}_N$, In particular $i_0{\in}{\cal A}^N_f(\tau_n)$ and therefore, by induction hypothesis, $i_0{\in}{\cal B}^N_f(\tau_n)$, so that we can couple both the duplication process for both processes $({\cal A}^N(t))$ and  $({\cal B}^N(t))$ as follows
      \begin{itemize}
      \item If $\card({\cal B}^N_f(\tau_n)){<}d$. A copy is made on the same node for both processes $({\cal A}^N(t))$ and $({\cal B}^N(t))$.
      \item If $\card({\cal B}^N_f(\tau_n)){=}d$. There exists some node $i_0$ such that $i_0{\not\in}{\cal A}^N_f(\tau_n)$ and $i_0{\in}{\cal B}^N_f(\tau_n)$. We can then set
        ${\cal A}^N_f(\tau_{n+1})={\cal A}^N_f(\tau_n){\cup}\{i_0\}$, remember that for the process $({\cal B}^N(t))$ the servers where to make a copy are also chosen at random. 
      \end{itemize}
      In both cases the relation  ${\cal A}^N_f(\tau_{n+1}){\subset} {\cal B}^N_f(\tau_{n+1})$ will hold.
      \item If a duplication occurs for the process $({\cal B}^N(t))$ at time $\tau_{n+1}$ but not for  the process $({\cal A}^N(t))$ then,  clearly,  the desired relation will then also hold  at time $\tau_{n+1}$.
  \end{enumerate}
  The lemma is proved.
\end{proof}
For the dominating model, for $k{\in}\{1,\ldots,d\}$ and $1{\leq}i{\leq}N$, we will denote by $T_{i,k}^N(t)$  the number of files of type $k$ and  with one copy on node $i$ at time $t{\geq}0$, this is the analogue of the variable  $R_{i,k}^N(t)$ defined above, 
\[
T_{i,k}^N(t)=\sum_{f{\in}{\cal F}_N} \ind{i\in {\cal B}_f^N(t), d^N_f(t)=k}. 
\]
For a given node $i$, if $(T_{i,k}^N(t{-})){=}r{=}(r_k)$, provided that there are no multiple copies on the same server just before time $t$, the transition rates of this process at time $t$ are given by
\[
r\mapsto  \begin{cases}
  r{-}e_{k}{+}e_{k+1},  &\text{\rm at rate } \lambda k r_k,\qquad 1{\leq} k{<}d,\\
  r{+}e_{k-1}{-}e_k, &\phantom {\rm at rate  aa}  \mu (k{-}1)r_k,\qquad  1{<}k{\leq}d.
\end{cases}
\]
and
\[
r\mapsto \begin{cases}
  (0,0), & \text {\rm at rate  } \mu,\\
 r{+}e_k,&\displaystyle \phantom {\rm at rate  aa}  \lambda\frac{(k{-}1)}{N{-}k{+}1}\sum_{f{\in}{\cal F}_N} \ind{i\not\in {\cal B}_f^N(t), d^N_f(t-)=k-1}.
\end{cases}
\]
Note that the last sum
is the sum of the terms $T_{j,k-1}^N(t{-})$, $j{=}1$, \ldots, $N$ minus some term which is less that $(k{-}1)T_{i,k-1}^N(t{-})$. The term $T_{i,k-1}^N(t{-})$, with appropriate estimates as  in Section~\ref{MF-sec},  will vanish in the limit when divided by  $N{-}k{+}1$. Consequently, asymptotically,  the transitions of the vector $(T_{i,k}^N(t))$ can be expressed as a functional of its coordinates.  With the same methods as for the original model for $d{=}2$ in Section~\ref{MF-sec}, it is not difficult to show that an analogue of Theorem~\ref{MF-Theo} holds.
\begin{theorem} (Mean-Field Convergence Theorem)\label{MF-Theo2}
  Suppose the process $({\cal A}^N(t))$ is initialized according to Assumption \ref{asm_initial2}. The  process of the empirical distribution
  \[
(\Lambda^N(t)) =\left(\frac{1}{N}\sum_{i=1}^N \delta_{(T_{i,k}^N(t),1\leq k\leq d)}\right)
  \]
  converges in distribution to a process $(\Lambda(t)\in{\cal D}(\R_+,{\cal P}(\N^2))$ such that :  for every function $g$ with finite support on $\N^d$,
\[
    \Lambda(t)(g)\stackrel{\text{def.}}{=} \Ept{g\left(\overline{T}_k(t)\right)},
\]
where $(\overline{T}(t)){=}(\overline{T}_k(t))$ is  a nonlinear Markov process  with the following transition rates: if $(\overline{T}(t))$ is in state $r{=}(r_k)$ just before  time $t$, the next possible state and the corresponding rates are given by
\[
r\mapsto \begin{cases}
  (0,0), & \mu,\\
  r{+}e_k,& \lambda \E\left(\overline{T}_{k-1}(t)\right),\, 1{\leq} k{\leq} d,
\end{cases}
\text{ and } \begin{cases}
  r{-}e_{k}{+}e_{k+1}, &\lambda k r_k,\, 1{\leq}k<d,\\
  r{+}e_{k-1}{-}e_k, & \mu (k{-}1)r_k,\, 1{<} k{\leq} d.
\end{cases}
\]
\end{theorem}
An argument similar to that in the proof of Theorem~\ref{McKeanTheo} shows the existence and  uniqueness of the Markov process $(\overline{T}(t))$. Note that the nonlinear component is now given by the vector of the mean values $\E(\overline{T}_k(t))$, $k{=}1$, \ldots, $d$.

The limiting Markov process $(\overline{T}(t)){=}(\overline{T}_k(t))$ can also be seen as the solution of the following SDEs, for $t\geq 0$
\begin{multline}\label{eqd1}
    \diff \overline{T}_{1}(t) =\displaystyle {-} \overline{T}_{1}(t{-}){\cal N}_{\mu}(\diff t) \displaystyle  {-} \int_{\R_+}\ind{0\le h\le\overline{T}_{1}(t-)}\overline{{\cal N}}_{\lambda}(\diff t,\diff h ) \\ {+} \int_{\R_+}\ind{0\le h\le\overline{T}_{2}(t-)}\overline{{\cal N}}_{\mu}(\diff t,\diff h ),
\end{multline}
for $1{<}k{<}d$,
\begin{multline}\label{eqd2}
\diff \overline{T}_{k}(t) =\displaystyle {-} \overline{T}_{k}(t{-}){\cal N}_{\mu}(\diff t)\\
{-} \int_{\R_+}\ind{0\le h\le\overline{T}_{k}(t-)}\overline{{\cal N}}_{k\lambda}(\diff t,\diff h ) +\int_{\R_+}\ind{0\le h\le\overline{T}_{k-1}(t-)}\overline{{\cal N}}_{(k-1)\lambda}(\diff t,\diff h )\\
-\int_{\R_+}\ind{0\le h\le\overline{T}_{k}(t-)}\overline{{\cal N}}_{(k-1)\mu}(\diff t,\diff h ) {+} \int_{\R_+}\ind{0\le h\le\overline{T}_{k+1}(t-)}\overline{{\cal N}}_{k\mu}(\diff t,\diff h )\\
+\int_{\R_+}\ind{0\le h\le\E(\overline{T}_{k-1}(t))}\overline{{\cal N}}_{\lambda,k-1}(\diff t,\diff h ),
\end{multline}
\begin{multline}\label{eqd3}
\diff \overline{T}_{d}(t) =\displaystyle {-} \overline{T}_{d}(t{-}){\cal N}_{\mu}(\diff t) +\int_{\R_+}\ind{0\le h\le\E(\overline{T}_{d-1}(t))}\overline{{\cal N}}_{\lambda,d-1}(\diff t,\diff h )\\
+\int_{\R_+}\ind{0\le h\le\overline{T}_{d-1}(t-)}\overline{{\cal N}}_{(d-1)\lambda}(\diff t,\diff h )
-\int_{\R_+}\ind{0\le h\le\overline{T}_{d}(t-)}\overline{{\cal N}}_{(d-1)\mu}(\diff t,\diff h ).
\end{multline}

The interesting property of these SDEs is that the vector of expected values can be expressed as the solution of a classical ODE, as the next proposition states.
\begin{prop}\label{propM}
For $t{\geq} 0$, the function $V(t){=}\E(\overline{T}_k(t/\mu))$ satisfies
\begin{equation}\label{eqdM}
\frac{\diff}{\diff t} V(t)=M_\rho{\cdot} V(t),
\end{equation}
with  
\[
M_\rho=\begin{pmatrix}
{-}(\rho{+}1)&1&0&& &0\\
2\rho& {-}2(\rho{+}1)& 2&0&&0\\
0& 3\rho& {-}3(\rho{+}1)& 3&0&0\\
\ldots& \ldots&\ldots&\ldots&\ldots&\ldots\\
0&0& k\rho&{-}k(\rho{+}1)&k&0\\
\ldots& \ldots&\ldots&\ldots&\ldots&\ldots\\
0& &&0&d \rho&-d\\
\end{pmatrix}
\]
and $\rho{=}\lambda/\mu$. Moreover, the matrix $M_\rho$ has $d$ distinct negative eigenvalues and the largest of them, ${-}\kappa_d^+(\rho)$, satisfies
\begin{equation}\label{upkd}
    0<\kappa_d^+(\rho)\leq \overline{\kappa}_d(\rho)\stackrel{\text{\rm def.}}{=}\left(\sum_{k=1}^d \frac{\rho^{k-1}}{k}\right)^{-1}<1.
\end{equation}
Finally, there exists a positive constant $K_0$ such that, for all $1\leq k\leq d$ and $t\geq 0$,
\begin{equation}\label{alt_decay}
\Ept{\overline{T}_k(t)}\leq K_0 e^{-\mu\kappa_d^+(\rho)t}.
\end{equation}
\end{prop}
\begin{proof}
Equation~\eqref{eqdM} can be obtained by taking the expected value of both sides of the integral version of Equations~\eqref{eqd1},~\eqref{eqd2} and~\eqref{eqd3}. For the next claim, since the matrix $M_\rho$ is a tridiagonal matrix, it has $d$ distinct real eigenvalues (see e.g.\ Chapter 1 of Fallat and Johnson~\cite{Fallat}).  If $D$ is the $d{\times}d$ diagonal  matrix whose $k$th diagonal component is ${ 1/\sqrt{k\rho^{k-1}}}$, then
\[
DM_\rho D^{-1}{=}\begin{pmatrix}
{-}(\rho{+}1)&\sqrt{2\rho}&0&& &0\\
\sqrt{2\rho}& {-}2(\rho{+}1)& \sqrt{6\rho}&0&&0\\
0& \sqrt{6\rho}& {-}3(\rho{+}1)& \sqrt{12\rho}&0&0\\
\ldots& \ldots&\ldots&\ldots&\ldots&\ldots\\
0&0& \sqrt{k(k{-}1)\rho}&{-}k(\rho{+}1)&\sqrt{k(k{+}1)\rho}&0\\
\ldots& \ldots&\ldots&\ldots&\ldots&\ldots\\
0& &&0&\sqrt{d(d{-}1) \rho}&-d\\
\end{pmatrix}
\]
is a symmetric matrix with the same eigenvalues as $M_\rho$. A straightforward calculation shows that its associated quadratic form is given by
\[
q(x_1,...,x_d)\stackrel{\text{\rm def.}}{=}-\sum_{i=1}^{d-1}\left(\sqrt{k\rho}x_k-\sqrt{(k{+}1)}x_{k+1}\right)^2-x_1^2, \quad (x_1,...,x_d)\in\R^d,
\]
which implies that all eigenvalues of $M_\rho$ are negative.   The  maximal eigenvalue of the symmetric matrix can be expressed as 
\[
\sup\left(q(y): y=(y_1,...,y_d)\in\R^d, \|y\|=1\right),
\]
with  $\|y\|^2{=}y_1^2{+}\cdots{+}y_d^2$, see e.g.\ page 176 of Horn and Johnson~\cite{Horn}. Taking the vector $x{=}(x_1,...,x_d)$ such that
\[
x_{k}{=}x_{k-1}\sqrt{\frac{k{-}1}{k}\rho},\quad 1<k\leq d,
\]
and $x_1$ is chosen so that $\|x\|{=}1$, one gets the upper bound~\eqref{upkd}.

Finally, Equation~\eqref{eqdM} shows that the components of $V(\cdot)$ can be expressed as a linear combination of the functions  $(\exp(\lambda_k\mu t))$, $1{\leq} k{\leq} d$. Since  all  eigenvalues of $M_\rho$ are negative,  ${-}\kappa_d^+$ is the largest eigenvalue, Relation~\eqref{alt_decay} follows.
\end{proof}

\medskip

\renewcommand{\theenumi}{\arabic{enumi}}
\noindent
  {\bf Remarks}
    \begin{enumerate}
\item We have  not been able to get a closed form expression for the actual exponential decay rate $\kappa_d^+(\rho)$ associated to the process $(\overline{T}_k(t))$. However, the upper bound ${\overline{\kappa}_d(\rho)}$ defined in Equation~\eqref{upkd} gives a lower bound for the decay rate. In Figure~\ref{Fig3}, we plot the ration $\bar\kappa_d(\rho)/\kappa^+_d(\rho)$ for different values of $\rho$ and $d$. 

\item Note that if the duplication rate $\lambda$ is larger than $\mu$, i.e.\ $\rho{>}1$, then
\[
\lim_{d\to+\infty} {{\kappa}_d^+(\rho)}=0.
\]
 \end{enumerate}
    
\begin{figure}[ht]
\centering
\includegraphics[scale=0.35]{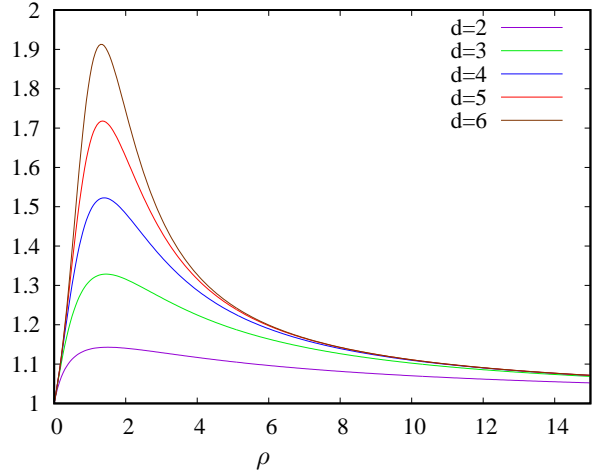}
\put(-150,15){$\rho$}
 \vspace{-5mm}
\caption{Accuracy of the upper bound of Relation~\eqref{upkd}: The ratio ${\overline{\kappa}_d(\rho)}/{\kappa_d^+(\rho)}$ for various values of $\rho$ and $d$.}\label{Fig3}
\end{figure}

Finally, for the case $d{=}2$, we can compare our result on the decay rate of file system with the decay rate of the process $(\E(\overline{T}_k(t)))$. 
\begin{corollary}\label{S-prop}
For $d{=}2$, if $(\overline{T}_1(0),\overline{T}_2(0)){=}(0,r_2)$, we have
\begin{align*}
\Ept{\overline{T}_1(t)} & =\frac{r_2}{\kappa_2^+(\rho)-\kappa_2^-(\rho)}\left(e^{-\mu\kappa_2^+(\rho) t}-e^{-\mu\kappa_2^-(\rho) t}\right)\\[2mm]
\Ept{\overline{T}_2(t)} & = \frac{r_2}{\kappa_2^+(\rho)-\kappa_2^-(\rho)}\left(y_+e^{-\mu\kappa_2^+(\rho)t}-y_-e^{-\mu\kappa_2^-(\rho) t}\right),
\end{align*}
where
\[
\kappa_2^\pm(\rho)\stackrel{\text{\rm def.}}{=}\frac{(3{+}\rho)\pm\sqrt{(3{+}\rho)^2{-}8}}{2}
\]
and  $y_\pm{\steq{def}}{-}\kappa_2^\pm{+}\rho{+}1.$
\end{corollary}
The definitions with $\pm$ just indicate that the identities are taken for $+$ and $-$ separately. Note that the definition of $\kappa_2^+(\rho)$ is consistent with the definition of Proposition~\ref{propM}.

\begin{proof}[Proof of Corollary \ref{S-prop}]
The above proposition can be used but the work has already been done to prove  Proposition~\ref{up-prop}. It is not difficult to show that
\[
\left({m}_1(t),{m}_2(t)\right)=\left(\E\left(\overline{T}_1(t)\right),  \E\left(\overline{T}_2(t)\right)\right)
\]
satisfies Relation~\eqref{ode}   with $p(t){=}{m}_1(t)$, i.e.\ so that $g(t){=}0$ with the notations of the  proof of Proposition~\ref{up-prop}. One concludes by using Relations~\eqref{m1} and~\eqref{m2}. 
\end{proof}

\subsection*{A Bound on the Exponential Decay Rate of the Algorithm}
The following proposition gives an estimation of the rate of decay of the network. 
\begin{prop}
  If $L^N(t)$ is the number of files alive at time $t$, 
  \[
  L^N(t)=\sum_{f\in{\cal F}_N}\ind{{\cal A}^N_f(t)\not=\emptyset},
  \]
then there exists a constant $K_1{>}0$ such that, for all $t{\geq}0$,
\[
    \limsup_{N\to+\infty} \E\left(\frac{L^N(t)}{N}\right)\leq K_1e^{-\mu\kappa_d^+(\rho) t},
\]
where $\kappa_d^+(\rho)$ is the constant defined in Proposition~\ref{propM}.
\end{prop}
\begin{proof}
  By using the coupling of the last proposition, one gets 
\begin{multline*}
  \E\left(L^N(t)\right)\leq \sum_{f\in{\cal F}_N}\P\left({\cal B}^N_f(t){\not=}\emptyset\right)
  \\  \leq \sum_{f\in{\cal F}_N}\E\left({\cal B}^N_f(t)\right)   =\sum_{i=1}^N \sum_{k=1}^d \frac{1}{k}\E\left(T_{i,k}^N(t)\right)
    =N \sum_{k=1}^d  \frac{1}{k}\E\left(T_{1,k}^N(t)\right).
\end{multline*}
Theorem~\ref{MF-Theo2} gives the convergence of the sequence of processes $(T_{1,k}^N(t))$ to the solution  of the EDS~\eqref{eqd1}, \eqref{eqd3} and~\eqref{eqd3}. It is not difficult to obtain an analogue of Lemma~\ref{Tech_set} which guarantee the boundedness of the second moments of the variables $T_{1,k}^N(t)$, $k{=}1$,\ldots,$d$, which gives the convergence of the first moments. One has obtained the relation,
\[
    \limsup_{N\to+\infty} \E\left(\frac{L^N(t)}{N}\right)\leq \sum_{k=1}^d  \frac{1}{k}\E\left(\overline{T}_{k}(t)\right),
    \]
one concludes with Inequality~\eqref{alt_decay}. The proposition is proved.     
\end{proof}

\providecommand{\bysame}{\leavevmode\hbox to3em{\hrulefill}\thinspace}
\providecommand{\MR}{\relax\ifhmode\unskip\space\fi MR }
\providecommand{\MRhref}[2]{%
  \href{http://www.ams.org/mathscinet-getitem?mr=#1}{#2}
}
\providecommand{\href}[2]{#2}

\end{document}